\documentclass[11pt]{article}

\usepackage{tikz}
\usepackage{subfigure}
\usepackage{scalefnt}

\usepackage{amsmath,amsfonts,amsthm,enumerate}
\usepackage{hyperref}

\numberwithin{equation}{section}

\newcommand{\TPSS}{S^{\hspace{.2mm}2} \mbox{$\times
\hspace{-2.8mm}_{-}$} \, S^{\hspace{.1mm}1}}

\newtheorem{defn}{{\bf Definition}}[section]
\newtheorem{eg}[defn]{{\bf Example}}
\newtheorem{lemma}[defn]{{\bf Lemma}}
\newtheorem{prop}[defn]{{\bf Proposition}}
\newtheorem{theo}[defn]{{\bf Theorem}}
\newtheorem{cor}[defn]{{\bf Corollary}}
\newtheorem{remark}[defn]{{\bf Remark}}

\begin{document}

\title{Minimal crystallizations of 3-manifolds}
\author{Biplab Basak and Basudeb Datta}

\date{}

\maketitle

\vspace{-10mm}
\begin{center}

\noindent {\small Department of Mathematics, Indian Institute of Science, Bangalore 560\,012, India.$^1$}


\footnotetext[1]{{\em E-mail addresses:} biplab10@math.iisc.ernet.in (B.
Basak), dattab@math.iisc.ernet.in (B. Datta).}

\medskip

\date{December 10, 2013}
\end{center}


\hrule

\begin{abstract}
We have introduced the weight of a group which has a presentation with number of relations is at most the number
of generators. We have shown that the number of facets of any contracted pseudotriangulation of a connected
closed 3-manifold $M$ is at least the weight of $\pi(M, \ast)$. This lower bound is sharp for the 3-manifolds
$\mathbb{R P}^3$, $L(3,1)$, $L(5,2)$, $S^1\times S^1 \times S^1$, $S^{\hspace{.2mm}2} \times S^1$, $\TPSS$ and
$S^{\hspace{.2mm}3}/Q_8$, where $Q_8$ is the quaternion group. Moreover, there is a unique such facet minimal
pseudotriangulation in each of these seven cases.  We have also constructed contracted pseudotriangulations of
$L(kq-1,q)$ with $4(q+k-1)$ facets for $q \geq 3$, $k \geq 2$ and $L(kq+1,q)$ with $4(q+k)$ facets for $q\geq 4$,
$k\geq 1$. By a recent result of Swartz, our pseudotriangulations of $L(kq+1, q)$ are facet minimal when $kq+1$
are even. In 1979, Gagliardi found presentations of the fundamental group of a manifold $M$ in terms of a
contracted pseudotriangulation of $M$. Our construction is the converse of this, namely, given a presentation of
the fundamental group of a 3-manifold $M$, we construct a contracted pseudotriangulation of $M$. So, our
construction of a contracted pseudotriangulation of a 3-manifold $M$ is based on a presentation of the
fundamental group of $M$ and it is computer-free.

\end{abstract}


\noindent {\small {\em MSC 2010\,:} Primary 57Q15. Secondary 57Q05, 57N10, 05C15.


}

\medskip

\hrule

\section{Introduction and Results}

A {\em simplicial cell complex} $K$ is the face poset of a regular CW complex $W$ such that the boundary complex
of each cell is isomorphic to the boundary complex of a simplex of same dimension. The space $W$ is said to be
the {\em geometric carrier} of $K$ and is denoted by $|K|$. If a space $X$ is homeomorphic to $|K|$ then we say
that $K$ is a {\em pseudotriangulation} of $X$. For $d \geq 1$, a graph $\Gamma = (V, E)$ with an edge coloring
$\gamma : E \to \{1, \dots, d+1\}$ determines a $d$-dimensional simplicial cell complex $\mathcal{K}(\Gamma)$
whose vertices have one to one correspondence with the colors $1, \dots, d+1$ and the facets have one to one
correspondence with the vertices in $V$. If $\mathcal{K}(\Gamma)$ is a pseudotriangulation of a space $M$ then
$(\Gamma, \gamma)$ is called a {\em crystallization} of $M$. So, if $(\Gamma, \gamma)$ is a crystallization of a
$d$-manifold $M$ then the number of vertices in the pseudotriangulation ${\mathcal K}(\Gamma)$ of $M$ is $d+1$.
In \cite{pe74}, Pezzana showed the following.

\begin{prop}[Pezzana] \label{prop:pezzana74}
Every connected closed PL-manifold admits a crystallization.
\end{prop}

Thus, every connected closed pl $d$-manifold has a {\em contracted} pseudotriangulation, i.e., a
pseudotriangulation with $d+1$ vertices. In this article, we are interested on crystallizations of connected
closed 3-manifolds with minimum number of vertices.

In \cite{ep61}, Epstein proved that the fundamental group of a 3-manifold has a presentation with the number of
relations less than or equal to the number of generators. For such a group $G$, let $\psi(G)$ be weight of $G$ as
in Definition \ref{def:weight} below. Then the weight of the trivial group is 2 and $\psi(G) \geq 8$ for any
nontrivial group $G$.

\begin{defn} \label{def:psi(M)}
{\rm For a connected closed 3-manifold $M$, let $\psi(M)$ be the {\em weight} $\psi(\pi(M, x))$ of the group
$\pi(M, x)$ for some $x$ in $M$. }
\end{defn}

If $M$ and $N$ are homeomorphic then clearly $\psi(M) = \psi(N)$. Thus, $\psi(M)$ is a topological invariant.
Clearly, $ \psi(S^{\hspace{.15mm}3}) =2$ and, in view of Perelman's theorem (Poincar\'{e} conjecture)
\cite{pe03}, $\psi(M) \geq 8$ for $M \neq S^{\hspace{.15mm}3}$.
Here, we have the following.

\begin{lemma} \label{lemma:psi(M)}
Let $\psi(M)$ be as above and let $Q_8$ be the quaternion group $\{\pm 1, \pm i, \pm j, \pm k\}$. Then
$\psi(\mathbb{R P}^3) = \psi(S^2 \times S^1)  = \psi(\TPSS)= 8$, $\psi(L(3,1)) = 12$, $\psi(L(5,q)) = 16$,
$\psi(S^{\hspace{.3mm}3}/Q_8)$ $= 18$, $\psi(S^1 \times S^1 \times S^1)
= 24$ for $1\leq q\leq 2$.
\end{lemma}

For a $d$-dimensional simplicial cell complex $K$, let $f_j(K)$ denote the number of $j$-cells of $K$ for $0\leq
j\leq d$. Let $m(K)$ be the minimal number of generators of $\pi(|K|, \ast)$. If ${\rm lk}_{K}(\alpha) :=
\{\beta\in K \, : \, \alpha \subseteq \beta\}$ (the {\em link} of $\alpha$) is connected for each cell $\alpha$
of codimension $\geq 2$ then $K$ is said to be a {\em normal pseudomanifold}. Clearly, a pseudotriangulation of a
connected manifold is a normal pseudomanifold. In \cite{ns09}, Novik and Swartz proved (Kalai's conjecture) that
$g_2(K) := f_1(K) -(d+1)f_0(K) + \binom{d+2}{2} \geq \binom{d+2}{2}\beta_1(K; \mathbb{F})$ for any
$\mathbb{F}$-orientable $\mathbb{F}$-homology manifold $K$ of dimension $d\geq 3$. Consider the contracted
pseudotriangulation $K_1 := \mathcal{K}(\mathcal{J}_1)$ of $S^2\times S^1$ corresponding to the crystallization
$\mathcal{J}_1$ in Fig. \ref{fig:J12} below. Since $f_3(K_1) =8$, it follows that $f_1(K_1)=12$. Therefore,
$g_2(K_1) = 12-16+10=6 < 10 = \binom{3+2}{2}\beta_1(S^2 \times S^1; \mathbb{Q}\hspace{.3mm})$. Thus, Novik-Swartz
theorem is not true for pseudotriangulations of manifolds. In \cite{kl09}, Klee proved that $h_2(K) := f_1(K)
-df_0(K) + \binom{d+1}{2} \geq \binom{d+1}{2}m(K)$ for any $d$-dimensional normal pseudomanifold $K$ whose edge
graph is $(d+1)$-colorable. For a connected closed pl $d$-manifold $M$, let $\Psi(M) := \min\{m \, : \, M$ has a
crystallization with $m$ vertices$\} = \min\{f_d(K) \, : \, K$ is a contracted pseudotriangulation of $M\}$. Here
we have the following.

\begin{theo} \label{theorem:T1}
Let $M$ be a connected closed $3$-manifold. If $(\Gamma, \gamma)$ is a crystallization of $M$ then $\Gamma$ has
at least $\psi(M)$ vertices. Equivalently, if $X$ is a contracted pseudotriangulation of $M$ then $f_3(X) \geq
\psi(M)$.
\end{theo}

\begin{cor} \label{cor:h2>=6m}
Let $M$ be a connected closed $3$-manifold $M$ and $\mathbb{F}$ be a field. If $X$ is a contracted
pseudotriangulation of $M$ then $g_2(X) = h_2(X) \geq \Psi(M) -2\geq \psi(M)-2 \geq 6m(M) \geq 6\beta_1(M;
\mathbb{F})$.
\end{cor}

From the complete enumeration (obtained by using high-powered computers) of crystallizations of prime 3-manifolds
with at most 30 vertices, we know $\Psi(M)$ for all closed prime 3-manifolds $M$ with $\Psi(M) \leq 30$ (cf.
\cite{cc08, li95}). Let $G$ be a group which has a presentation with number of relations is at most the number of
generators. From Theorem \ref{theorem:T1} we know that the number of vertices in any crystallization  $(\Gamma,
\gamma)$ of a closed connected 3-manifold $M$, whose fundamental group is $G$, has at least $\psi(G)$ vertices.
To construct such a crystallization, we choose a presentation $\langle S, R\rangle$ of $G$ such that
$\varphi(S,R) = \psi(G, \rho(G))$. We then construct $(\Gamma, \gamma)$ which yields the presentation $\langle S,
R\rangle$ as mentioned at the end of Section \ref{crystal}. We have considered the groups $\mathbb{Z}$,
$\mathbb{Z}_2$, $\mathbb{Z}_3$, $\mathbb{Z}_5$, $\mathbb{Z}^3$ and $Q_8$ and have obtained such crystallizations.
Generalizing some of these constructions, we have constructed two infinite families of crystallizations of lens
spaces. More explicitly, we have the following.

\begin{theo} \label{theorem:T2}
Let $\psi(M)$ and $\Psi(M)$ be as above.
\begin{enumerate}[{\rm (i)}]
\item If $M = \mathbb{R P}^3$, $S^2 \times S^1$, $\TPSS$, $L(3,1)$,
$L(5,2)$, $S^{\hspace{.3mm}3}/Q_8$ or $S^1
\times S^1 \times S^1$ then $\Psi(M) = \psi(M)$ and $M$ has a unique contracted pseudotriangulation with
$\psi(M)$ facets.

\item Let $X$ be a contracted pseudotriangulation of a connected closed $3$-manifold $M$. If $f_3(X) \leq 8$ then
$M$ is $($homeomorphic to$)$ $S^3$, $\mathbb{R P}^3$, $S^2 \times S^1$ or $\TPSS$.
\end{enumerate}
\end{theo}

\begin{cor} \label{cor:h2>6m}
Let $X$ be a contracted pseudotriangulation of a closed $3$-manifold
$M$. If $M$ is $S^{\hspace{.3mm}3}/Q_8$, $S^1 \times S^1 \times S^1$
or $L(p,q)$ for some $p\geq 3$ then $h_2(X) > 6m(M)$.
\end{cor}

\begin{theo} \label{theorem:T4}
Let $\Psi(M)$ be as above. Then
\begin{enumerate}[{\rm (i)}]
\item $\Psi(L(kq-1,q)) \leq 4(k+q-1)$ for $k, q\geq 2$ and \item  $\Psi(L(kq+1,q)) \leq 4(k+q)$ for $k, q\geq 1$.
\end{enumerate}
\end{theo}

\begin{remark} \label{remark:swartz}
{\rm Recently, Swartz proved that $\Psi(L(kq+1,q)) \geq 4(k+q)$ whenever $k, q$ are odd (\cite{sw13}). Thus,
$\Psi(L(kq+1,q)) = 4(k+q)$ for odd positive integers $k, q$. We found that $\Psi(L(5,1))= 20 = \Psi(L(7,2))$. So,
Swartz's bound is also valid for $L(5,1)$ and $L(7,2)$. We also found that $\psi(\mathbb{Z}_4) =14$ and
$\psi(\mathbb{Z}_6) = \psi(\mathbb{Z}_7) =18$. Proofs of these are in earlier versions of this article in the
arXiv (arXiv:1308.6137). We have omitted these proofs from this version for the sake of brevity.}
\end{remark}

\section{Preliminaries}

\subsection{Coloured Graphs}

All graphs considered here are finite multigraphs without loops. If $\Gamma = (V, E)$ is a graph and $U \subseteq
V$ then the {\em induced} subgraph $\Gamma[U]$ is the subgraph of $\Gamma$ whose vertex set is $U$ and edges are
those edges of $\Gamma$ whose end points are in $U$. For $n\geq 2$, an $n$-cycle is a closed path with $n$
distinct vertices and $n$ edges. If vertices $a_i$ and $a_{i+1}$ are adjacent in an $n$-cycle for $1\leq i\leq n$
(addition is modulo $n$) then the $n$-cycle is denoted by $C_n(a_1, a_2, \dots, a_n)$.

An {\em edge coloring} of  a graph $\Gamma = (V, E)$ is a map $\gamma \colon E \to C$ such that $\gamma(e) \neq
\gamma(f)$ whenever $e$ and $f$ are adjacent (i.e., $e$ and $f$ are adjacent to a common vertex). The elements of
$C$ are called the {\em colors}. If $C$ has $h$ elements then $(\Gamma, \gamma)$ is said to be an {\em
$h$-colored graph}.

Let $(\Gamma,\gamma)$ be an $h$-colored graph with color set $C$. If $B \subseteq C$ with $k$ elements then the
graph $(V(\Gamma), \gamma^{-1}(B))$ is a $k$-colored graph with coloring $\gamma|_{\gamma^{-1}(B)}$. This colored
graph is denoted by $\Gamma_B$. Let $(\Gamma,\gamma)$ be an $h$-colored connected graph with color set $C$. If
$\Gamma_{C\setminus\{c\}}$ is connected for all $c\in C$ then  $(\Gamma,\gamma)$ is called {\em contracted}.

Let $\Gamma_1 = (V_1, E_1)$ and $\Gamma_2=(V_2, E_2)$ be two disjoint $h$-regular $h$-colored graphs with same
color set $\{1, \dots, h\}$. For $1\leq i\leq 2$, let $v_i \in V_i$. Consider the graph $\Gamma$ which is
obtained from $(\Gamma_1 \setminus\{v_1\}) \sqcup (\Gamma_2 \setminus \{v_2\})$ by adding $h$ new edges $e_1,
\dots, e_h$ with colors $1, \dots, h$ respectively such that the end points of $e_j$ are $u_{j,1}$ and $u_{j,
2}$, where $v_i$ and $u_{j,i}$ are joined in $\Gamma_i$ with an edge of color $j$ for $1\leq j\leq h$, $1\leq
i\leq 2$. (Here $\Gamma_i \setminus\{v_i\} = \Gamma_i[V_i\setminus\{v_i\}]$.) The colored graph $\Gamma$ is
called the {\em connected sum} of $\Gamma_1$, $\Gamma_2$ and is denoted by $\Gamma_1\#_{v_1v_2}\Gamma_2$.

Let $\Gamma = (V, E)$ be a $(d+1)$-regular graph with a $(d+1)$-coloring $\gamma \colon E \to C$. Let $x, y\in V$
be joined by $k$ edges $e_1, \dots, e_k$, where $1\leq k\leq d$. Let $B = C\setminus \gamma(\{e_1, \dots,
e_k\})$. Let $X$ (resp., $Y$) be the components of $\Gamma_B$ containing $x$ (resp., $y$). If $X \neq Y$ then
$\Gamma[\{x, y\}]$ is called a {\em $d$-dimensional dipole of type $k$}. Dipoles of types 1 and $d$ are called
{\em degenerate} dipoles.

Let $\Gamma=(V,E)$ be a $(d+1)$-regular graph with a $(d+1)$-coloring $\gamma \colon E \to C$ and a dipole
$\Gamma[\{x, y\}]$ of type $k$. Let $B$, $X$ and $Y$ be as above. A $(d+1)$-regular graph $(\Gamma^{\prime},
\gamma^{\prime})$ with same color set $C$ is said to {\em obtained from $\Gamma$ by cancelling the dipole
$\Gamma[\{x, y\}]$} if (i) $\Gamma^{\prime}_B$ is obtained from $\Gamma_B$ by replacing $X\sqcup Y$ by $X\#_{xy}
Y$, and (ii) two vertices $u, v$ of $\Gamma^{\prime}$ are joined by an edge of  color $c\in B$ if and only if the
corresponding vertices of $\Gamma$ are  so (cf. \cite{fg82}). For standard terminology on graphs see \cite{bm08}.

\subsection{Presentation of Groups}

Given a set $S$, let $F(S)$ denote the free group generated by $S$. So, any element $w$ of $F(S)$ is of the form
$w = x_1^{\varepsilon_1} \cdots x_m^{\varepsilon_m}$, where $x_1, \dots, x_m\in S$ and $\varepsilon_i = \pm 1$
for $1\leq i\leq m$ and $(x_{j+1}, \varepsilon_{j+1}) \neq (x_j, -\varepsilon_j)$ for $1\leq j \leq m-1$. For $R
\subseteq F(S)$, let $N(R)$ be the smallest normal subgroup of $F(S)$ containing $R$. Then the quotient group
$F(S)/N(R)$ is denoted by $\langle S \, | \, R \rangle$. So, $\langle S \, | \, T \rangle = \langle S \, | \, R
\rangle$ if $N(T) = N(R)$. We write $\langle S_1 \, | \, R_1\rangle = \langle S_2 \, | \, R_2\rangle$ only when
$F(S_1) = F(S_2)$ and $N(R_1) = N(R_2)$. For $w_1, w_2\in F(S)$, if $w_1N(R) = w_2N(R)\in \langle S \, | \,
R\rangle$ then we write $w_1 \equiv w_2$ (mod $R$). Two elements $w_1, w_2\in F(S)$ are said to be {\em
independent} (resp., {\em dependent}) if $N(\{w_1\}) \neq N(\{w_2\})$ (resp., $N(\{w_1\}) = N(\{w_2\})$).

For a finite subset $R$ of $F(S)$, let
\begin{align}
\overline{R} := \{w \in N(R) \, : \, N((R \setminus \{r\}) \cup \{w\}) = N(R) \mbox{ for each } r\in R\}.
\end{align}
Observe that $\overline{\emptyset} = \emptyset$ and if $R\neq \emptyset$ is a finite set then $w :=\prod_{r\in R}
r \in \overline{R}$ and hence $\overline{R} \neq \emptyset$. Also, $\{wrw^{-1}, wr^{-1}w^{-1} \, : \, w\in F(S)\}
\subseteq \overline{\{r\}}$ for $r\in F(S)$.

For $w = x_1^{\varepsilon_1}\cdots x_m^{\varepsilon_m} \in F(S)$, $m\geq 1$, let
\begin{align}
\varepsilon(w) :=
\left\{ \begin{array}{lcl}
0 & \mbox{if} & m = 1, \\
|\varepsilon_1-
\varepsilon_2| + \cdots + |\varepsilon_{m-1}-\varepsilon_m| +
|\varepsilon_m - \varepsilon_1| & \mbox{if}
& m \geq 2.
\end{array}\right. \nonumber
\end{align}
Consider the map $\lambda \colon F(S) \to \mathbb{Z}^{+}$ define inductively as follows.
\begin{eqnarray}\label{l(w)}
\lambda(w) := \left\{ \begin{array}{lcl}
2 & \mbox{if} & w = \emptyset, \\
2m-\varepsilon(w) & \mbox{if}
& w = x_1^{\varepsilon_1}\cdots
x_m^{\varepsilon_m}, \, (x_m, \varepsilon_m) \neq (x_1, -\varepsilon_1), \\
\lambda(w^{\hspace{.1mm}\prime})
& \mbox{if} & w = x_1^{\varepsilon_1} w^{\hspace{.1mm}\prime} x_1^{-\varepsilon_1}.
\end{array}\right.
\end{eqnarray}
Since $|\varepsilon_i - \varepsilon_j|=0$ or 2, $\varepsilon(w)$
is an even integer and hence $\lambda(w)$ is
also even. For $w\in F(S)$, $\lambda(w)$ is said to be the {\em weight} of $w$. Observe that $\lambda(w_1w_2) =
\lambda(w_2w_1)$ for $w_1, w_2\in F(S)$.

Let $S = \{x_1, \dots, x_s\}$ and $R= \{r_1, \dots, r_t\} \subseteq F(S)$,  where $t\leq s$. Let $r_{t+1}$ be an
element in $\overline{R}$ of minimum weight. Let
\begin{align} \label{varphi(S,R)}
\varphi(S,R) := \lambda(r_1) +
\cdots + \lambda(r_t) + \lambda(r_{t+1}) + 2(s-t).
\end{align}

For a finitely presented group $G$ and a non-negative integer $q$, we define
\begin{align}
{\mathcal P}_q(G) := \{\langle S \, | \, R\rangle \cong G \, : \#(R)
\leq \#(S) \leq q\}. \nonumber
\end{align}
For a finitely presented group $G$, let $m(G)$ be the minimum number of
generators of $G$.
Here, we are interested on those groups $G$ for which ${\mathcal P}_q(G)\neq \emptyset$ for some $q$. Let
\begin{subequations}
\begin{align}
\mu(G) & := \min\{q \, : \, {\mathcal P}_q(G)\neq\emptyset\},\label{mu(G)}\\
\psi(G; q) & :=   \min\{\varphi(S, R) \, : \,  \langle S \, | \, R\rangle
\in {\mathcal P}_q(G)\}
\mbox{ for } q \geq \mu(G).   \label{psi(G;q)}
\end{align}
\end{subequations}
Clearly, $\mu(G) \geq m(G)$ and $\psi(G,q) \leq \psi(G,\mu(G))$ for all $q\geq \mu(G)$. Let
\begin{align}
\rho(G) & :=  \min\{q \geq \mu(G) \, : \, \psi(G;q) \leq 6(q+1)\}. \label{rho(G)}
\end{align}
So, $\rho(G)$ is the smallest integer $q$ such that $\psi(G; q) \leq 6(q+1)$.

\begin{defn} \label{def:weight}
{\rm Let $G$ be a group which has a presentation with the number of relations less than or equal to the number of
generators. Let $\mu(G)$, $\psi(G; q)$ and $\rho(G)$ be as above. Then $\psi(G) =  \max\{\psi(G; \rho(G)),
6\mu(G) +2\}$ is a positive even integer. The integer $\psi(G)$ is said to be the {\em weight} of the group $G$.
}
\end{defn}

\begin{remark} \label{remark:R0}
{\rm Observe that $\min\{\varphi(S, R) \, : \,
\langle S \, | \, R\rangle \cong \mathbb{Z}, \, \#(R)
\leq \#(S) < \infty \} = 4 =
\psi(\mathbb{Z}, \rho(\mathbb{Z})) < 8 =
\psi(\mathbb{Z})$ (see the proof of Lemma \ref{lemma:psi(M)}). In general, we have  $\min\{\varphi(S, R) \, : \,
\langle S \, | \, R\rangle \cong G, \, \#(R)
\leq \#(S) < \infty \} =
\min\{\min\{\varphi(S, R) \, : \,
\langle S \, | \, R\rangle \in \mathcal{P}_q(G)\}\, : \, \mu(G) \leq q <\infty\}
= \min\{\psi(G;q)\, : \,\mu(G) \leq  q <\infty\} \leq
\psi(G; \rho(G)) \leq \psi(G)$.
}
\end{remark}

\subsection{Lens Spaces}

Consider the 3-sphere $S^{\hspace{.2mm}3} = \{(z_1, z_2) \in \mathbb{C}^{\hspace{1mm}2} \, : \, |z_1|^2 + |z_2|^2
= 1\}$. Let $p$ and $q$ be relatively prime integers. Then the action of $\mathbb{Z}_p = \mathbb{Z}/p\mathbb{Z}$
on $S^3$ generated by $e^{2\pi i/p}\cdot(z_1, z_2) = (e^{2\pi i/p}z_1, e^{2\pi iq/p}z_2)$ is free and hence
properly discontinuous. Therefore the quotient space $L(p, q) := S^{\hspace{.3mm}3}/\mathbb{Z}_p$ is a 3-manifold
whose fundamental group is isomorphic to $\mathbb{Z}_p$. The 3-manifolds $L(p, q)$ are called the lens spaces. It
is a classical theorem of Reidemeister that $L(p, q^{\hspace{.3mm}\prime})$ is homeomorphic to $L(p, q)$ if and
only if $q^{\hspace{.3mm}\prime} \equiv \pm q^{\pm 1}$ (mod $p$).

If $T_1$, $T_2$ are two solid tori (i.e., each $T_j$ is homeomorphic to $\{(z, w) \in \mathbb{C}^{\hspace{1mm}2}
\, : \, |z|=1, \, |w|\leq 1\}$) such that (i) $T_1 \cap T_2 = \partial (T_1) = \partial (T_2) \cong S^1 \times S^1$,
(ii) $\pi_1(T_1\cap T_2, x) = \langle \alpha, \beta \, | \, \alpha\beta\alpha^{-1}\beta^{-1}\rangle$, (iii)
$\pi_1(T_1, x) = \langle \alpha\rangle$, (iv) $\pi_1(T_2, x) = \langle \alpha, \beta \, | \, \alpha\beta
\alpha^{-1}\beta^{-1}, \alpha^p\beta^q\rangle$ ($=\langle \alpha^m\beta^n\rangle$, where $m, n\in\mathbb{Z}$ such
that $mq-np=1$), for $x\in T_1\cap T_2$, then $T_1\cup T_2$ is homeomorphic to $L(p, q)$.

\subsection{Crystallizations} \label{crystal}

A CW-complex is said to be {\em regular} if all its closed cells are homeomorphic to closed balls.  Given a
CW-complex $X$, let ${\mathcal X}$ be the set of all closed cells of $X$ together with the empty set. Then
${\mathcal X}$ is a poset, where the partial ordering is the set inclusion. This poset ${\mathcal X}$ is said to
be the {\em face poset} of $X$. Clearly, if $X$ and $Y$ are two finite regular CW-complexes with isomorphic face
posets then $X$ and $Y$ are homeomorphic. A regular CW-complex $X$ is said to be {\em simplicial} if the boundary
of each cell in $X$ is isomorphic (as a poset) to the boundary of a simplex of same dimension. A {\em simplicial
cell complex} ${\mathcal X}$ of dimension $d$ is the face poset of a $d$-dimensional simplicial CW-complex $X$.
The topological space $X$ is called the {\em geometric carrier} of ${\mathcal X}$ and is also denoted by
$|{\mathcal X}|$. If a topological space $M$ is homeomorphic to $X$, then ${\mathcal X}$ is said to be a {\em
pseudotriangulation} of $M$ (see \cite{bj84, mu13} for more).

The 0-cells in a simplicial cell complex are said to be the vertices. Clearly, a $d$-dimensional simplicial cell
complex ${\mathcal X}$ has at least $d+1$ vertices. If  a $d$-dimensional simplicial cell complex ${\mathcal X}$
has exactly $d+1$ vertices then ${\mathcal X}$ is called {\em contracted}. If all the maximal cells of a
$d$-dimensional simplicial cell complex are $d$-cells then it is called {\em pure}.

Let ${\mathcal X}$ be a pure $d$-dimensional simplicial cell complex. Consider the graph $\Lambda({\mathcal X})$
whose vertices are the facets of ${\mathcal X}$ and edges are the ordered pairs $(\{\sigma_1, \sigma_2\},
\gamma)$, where $\sigma_1$, $\sigma_2$ are facets, $\gamma$ is a $(d-1)$-cell and is a common face of $\sigma_1$,
$\sigma_2$. The graph $\Lambda({\mathcal X})$ is said to be the {\em dual graph} of ${\mathcal X}$. Observe that
$\Lambda({\mathcal X})$ is in general a multigraph without loops. On the other hand, for $d\geq 1$, if $(\Gamma,
\gamma)$ is a $(d+1)$-colored connected graph with color set $C = \{1, \dots,  {d+1}\}$ then we define a
$d$-dimensional simplicial cell complex ${\mathcal K}(\Gamma)$ as follows. For each $v\in V(\Gamma)$ we take a
$d$-simplex $\sigma_v$ and label its vertices by ${1}, \dots, {d+1}$. If $u, v \in V(\Gamma)$ are joined by an
edge $e$ and $\gamma(e) =  {i}$, then we identify the $(d-1)$-faces of $\sigma_u$ and $\sigma_v$ opposite to the
vertices labelled by ${i}$, so that equally labelled vertices are identified together. Since there is no
identification within a $d$-simplex, this gives a simplicial CW-complex $W$ of dimension $d$. So, the face poset
(denoted by ${\mathcal K}(\Gamma)$) of $W$ is a pure $d$-dimensional simplicial cell complex. We say that
$(\Gamma, \gamma)$ {\em represents} the simplicial cell complex ${\mathcal K}(\Gamma)$. Clearly, the number of
$i$-labelled vertices of ${\mathcal K}(\Gamma)$ is equal to the number of components of $\Gamma_{C\setminus\{
{i}\}}$ for each $ {i}\in C$. Thus, the simplicial cell complex ${\mathcal K}(\Gamma)$ is contracted if and only
if $\Gamma$ is contracted  (cf. \cite{fgg86}).

A {\em crystallization} of a connected closed $d$-manifold $M$ is a $(d+1)$-colored contracted graph $(\Gamma,
\gamma)$ such that the simplicial cell complex ${\mathcal K}(\Gamma)$ is a pseudotriangulation of $M$. Thus, if
$(\Gamma, \gamma)$ is a crystallization of a $d$-manifold $M$ then the number of vertices in ${\mathcal
K}(\Gamma)$ is $d+1$. On the other hand, if $K$ is a contracted pseudotriangulation of $M$ then the dual
graph $\Lambda(K)$ gives a crystallization of $M$. Clearly, if $(\Gamma, \gamma)$ is a crystallization of a closed $d$-manifold $M$ then either $\Gamma$ has two vertices (in which case $M$ is $S^d$) or the number of edges between two vertices is at most $d-1$. From  \cite{cgp80}, we know the following.

\begin{prop}[Cavicchioli-Grasselli-Pezzana] \label{prop:ca-ga-pe80}
Let $(\Gamma,\gamma)$ be a crystallization of an $n$-manifold $M$. Then $M$ is orientable if and only if $\Gamma$
is bipartite.
\end{prop}

For $k\geq 2$, let $1, \dots, k$ be the colors of a $k$-colored graph $(\Gamma, \gamma)$. For $1\leq i\neq j\leq
k$, $\Gamma_{ij}$ denote the graph $\Gamma_{\{i, j\}}$ and $g_{ij}$ denote the number of connected components of
the graph $\Gamma_{ij}$. In \cite{ga79a}, Gagliardi proved the following.

\begin{prop}[Gagliardi] \label{prop:gagliardi79a}
Let $(\Gamma,\gamma)$ be a contracted $4$-colored graph with $m$ vertices. Then $(\Gamma,\gamma)$ is a
crystallization of a connected closed $3$-manifold if and only if
\begin{enumerate}[{\rm (i)}]
\item
$g_{ij}=g_{kl}$ for every permutation $ijkl$ of $1234$, and
\item $g_{12}+g_{13}+g_{14}=2+m/2$.
\end{enumerate}
\end{prop}

Let  $(\Gamma, \gamma)$ be a crystallization (with the color set $C$) of a connected closed $n$-manifold $M$. So,
$\Gamma$ is an $(n+1)$-regular graph. Choose two colors, say, $i$ and $j$ from $C$. Let $\{G_1, \dots, G_{s+1}\}$
be the set of all connected components of $\Gamma_{C\setminus \{i,j\}}$ and $\{H_1, \dots, H_{t+1}\}$ be the set
of all connected components of $\Gamma_{ij}$. Since $\Gamma$ is regular, each $H_p$ is an even cycle. Note
that, if $n=2$, then $\Gamma_{ij}$ is connected and hence $H_1= \Gamma_{ij}$. Take a set $\widetilde{S} =
\{x_1, \dots, x_s, x_{s+1}\}$ of $s+1$ elements. For $1\leq k\leq t+1$, consider the word $\tilde{r}_k$ in
$F(\widetilde{S})$ as follows. Choose a vertex $v_1$ in $H_k$. Let $H_k = v_1 e_{1}^i v_2 e_{2}^j v_3 e_{3}^i
v_{4} \cdots e_{2l-1}^i v_{2l}e_{2l}^jv_1$, where $ e_{p}^i$ and  $ e_{q}^j$ are edges with colors $i$ and $j$
respectively. Define
\begin{align} \label{tildar}
\tilde{r}_k := x_{k_2}^{+1} x_{k_3}^{-1}x_{k_4}^{+1}  \cdots
x_{k_{2l}}^{+1}x_{k_1}^{-1},
\end{align}
where $G_{k_h}$ is the component of $\Gamma_{C\setminus \{i,j\}}$ containing $v_h$.  For $1\leq k\leq t+1$, let
$r_k$ be the word obtained from $\tilde{r}_k$ by deleting $x_{s+1}^{\pm 1}$'s in $\tilde{r}_k$. So, $r_k$ is a
word in $F(S)$, where $S = \widetilde{S}\setminus \{x_{s+1}\}$. In \cite{ga79b}, Gagliardi proved the following.

\begin{prop}[Gagliardi] \label{prop:gagliardi79b}
For $n\geq 2$, let  $(\Gamma, \gamma)$ be a crystallization of a connected closed $n$-manifold $M$. For two
colors $i, j$, let $s$, $t$, $x_p$, $r_q$ be as above. If $\pi_1(M, x)$ is the fundamental group of $M$ at a
point $x$, then
$$
\pi_1(M, x) \cong \left\{ \begin{array}{lcl}
\langle {x_1, x_2,\dots, x_s} ~ | ~ {r_1} \rangle & \mbox{if}
& n=2,   \\
\langle {x_1, x_2, \dots, x_s} ~ | ~ {r_1, \dots, r_t} \rangle
& \mbox{if} & n\geq 3.
\end{array}\right.
$$
\end{prop}


\section{Proofs of Lemma \ref{lemma:psi(M)}, Theorem \ref{theorem:T1} and Corollary \ref{cor:h2>=6m}}

\noindent {\em Proof of Lemma} \ref{lemma:psi(M)}.
The result follows from the following lemma. \hfill $\Box$

\begin{lemma} \label{lemma:psi(G)}
{\rm (i)} $\psi(\mathbb{Z})  = \psi(\mathbb{Z}_2)= 8$,
{\rm (ii)} $\psi(\mathbb{Z}_3) = 12$, {\rm (iii)} $\psi(\mathbb{Z}_5) =
16$, {\rm (iv)}  $\psi(Q_8) = 18$ and {\rm (v)} $\psi(\mathbb{Z}^3) =
24$.
\end{lemma}

\begin{proof}
Any presentations of $\mathbb{Z}$ must have at least one generator and  $\langle x \rangle$ is a presentation of
$\mathbb{Z}$. So, $\mu(\mathbb{Z})=1$. If $\langle S | R\rangle \cong \mathbb{Z}$ with $\#(S)
=1$, then $R =\emptyset$ and hence, by the definition (see \eqref{varphi(S,R)}), $\varphi(S,R)= \lambda(\emptyset) + 2(1-0) = 2+2=4 <
12 = 6(\mu(\mathbb{Z})+1)$. Therefore, $\psi(\mathbb{Z};q) \leq 4$ for all $q\geq 1$. Thus, $\psi(\mathbb{Z}) =
\max\{\psi(\mathbb{Z}, \rho(\mathbb{Z})), 6\mu(\mathbb{Z}) +2\} = \max\{\psi(\mathbb{Z}, \rho(\mathbb{Z})), 8\}
=8$.

Let $p \geq 2$ be an integer. Since any presentations of $\mathbb{Z}_{p}$ must have at least one generator and $\langle x \, | \, x^{p} \rangle$ is a presentation of $\mathbb{Z}_{p}$, it follows that $\mu(\mathbb{Z}_{p})
=1$. Clearly, if $\langle S=\{x\} \,| \, R=\{r_1\} \rangle$ is a presentation of $\mathbb{Z}_p$, then $r_1 = x^{\pm
p}$. Let $r_2 \in \overline{R}$ be of minimum weight. Since $\langle x \, | r_2\rangle$ is also a presentation
of $\mathbb{Z}_{p}$, $r_2 = x^{\pm p}$. Therefore, by (\ref{varphi(S,R)}),
\begin{align} \label{eq:4p}
\varphi(S,R) = \lambda (r_1) +\lambda (r_2) = (2p - \varepsilon(r_1)) + (2p -\varepsilon(r_2)) = 4p.
\end{align}

First assume that $p\leq 3$. Since, $\langle S \, | \, R\rangle \in {\mathcal P}_1(\mathbb{Z}_p)$ implies (up to
renaming) $(S, R) = (\{x\}, \{x^p\})$ or $(\{x\}, \{x^{-p}\})$, it follows that $\psi(\mathbb{Z}_{p}; 1)=
\varphi(\{x\}, \{x^{\pm p}\}) = 4p \leq 12 = 6(\mu(\mathbb{Z}_{p}) +1)$. This implies that $\rho(\mathbb{Z}_{p})
= \mu(\mathbb{Z}_{p})=1$. Thus, $ \psi(\mathbb{Z}_{p};\rho(\mathbb{Z}_{p})) =4p \geq 8 =6\mu(\mathbb{Z}_{p})+2$.
Therefore, $\psi(\mathbb{Z}_{p})=4p$. This proves parts (i) and (ii).

\smallskip

Now, assume $p=5$. By the similar arguments as for $p \leq 3$, $\langle S \, | \, R\rangle \in {\mathcal
P}_1(\mathbb{Z}_5)$ implies $\varphi(S, R)=4p=20$. Therefore, $\psi(\mathbb{Z}_{5};1)=20 > 12 =
6(\mu(\mathbb{Z}_{5}) +1)$ and hence $\rho(\mathbb{Z}_{5})>\mu(\mathbb{Z}_{5})=1$. If we take $S =\{x_1, x_2\}$
and $R=\{r_1= x_1^2 x_2^{-1}, r_2= x_2^3x_1^{-1}\}$ then $\varphi(S,R)\leq 16$ (since $r_3 = x_1x_2^2\in
\overline{R}$ is of weight 6) and $\langle S \, | \, R \rangle \in {\mathcal P}_2(\mathbb{Z}_5) \setminus
{\mathcal P}_1(\mathbb{Z}_5)$. Thus, $\psi(\mathbb{Z}_{5};2) \leq 16<18 = 6(2+1)$. Therefore,
$\rho(\mathbb{Z}_{5})= 2$ and hence $\psi(\mathbb{Z}_{5})\leq 16$.

Now, let $\langle S \, | \, R \rangle \in {\mathcal P}_2(\mathbb{Z}_5) \setminus {\mathcal P}_1(\mathbb{Z}_5)$
with $\varphi(S, R)\leq 16$. Since there is no presentation $\langle S \, | \, R\rangle$ of $\mathbb{Z}_5$ with
$(\#(S), \#(R)) =(2,1)$, it follows that $\#(R)=\#(S)=2$. Let $S = \{ x_1, x_2 \}$ and $R=\{r_1,  r_2\}$. If
$\lambda(r_1) =2$, then $r_1$ must be of the form $x_i^{\pm 1}$ or $x_i^{\varepsilon}x_j^{-\varepsilon}$ for some
$j\neq i\in \{1, 2\}$ and $\varepsilon = \pm 1$. Since $\langle S \, | \, R \rangle \cong \mathbb{Z}_{5} $, it
follows that $r_2 \equiv x_j^{\pm 5}$ (mod $\{r_1\}$). This implies that $\lambda(r_2)\geq \lambda(x_j^{\pm 5})=
10$. Let $r_3 \in \overline{R}$ be of minimum weight. Then $\langle x_1, x_2 \, | \, r_1, r_3\rangle$ is also a
presentation of $ \mathbb{Z}_{5}$ and hence (by the same arguments) $\lambda(r_3)\geq 10$. Thus,  $\varphi(S, R)
= \lambda(r_1) + \lambda(r_2) + \lambda(r_3) \geq 2+10+10=22$, a contradiction. So, $\lambda(r_i) \geq 4$ for
$1\leq i\leq 2$. Let $A=\{x_1x_2, x_1^2, x_2^2, x_1^2x_2^{-1}, x_2^2x_1^{-1}, x_1x_2^{-1}x_1x_2^{-1}\}$ and let
$A^{-1}=\{w^{-1} \, : \, w\in A\}$. Then $A$ is a set of pairwise independent elements of weight $4$ in $F(S)$
and $w\in F(S)$ is an element of weight 4 imply that $w$ is dependent with an element of $A$. Note that
$\mathbb{Z}_{5}$ has no presentation $\langle S \, | \,  R\rangle \in {\mathcal P}_2(\mathbb{Z}_5) \setminus
{\mathcal P}_1(\mathbb{Z}_5)$ with $R\subseteq A \cup A^{-1}$. So, at most one of $r_1$, $r_2$, $r_3$ has weight
4 and the weights of other two are at least 6. Therefore, $\varphi(S,R) \geq 16$. This implies that
$\psi(\mathbb{Z}_{5})= 16$. This proves part (iii).

\smallskip

Clearly, $\mu(Q_8)= 2$. If we take $S=\{x_1, x_2\}$ and $R=\{x_2x_1x_2x_1^{-1}, x_1x_2x_1x_2^{-1} \}$ then
$\langle S \, | \, R\rangle\in \mathcal{P}_2(Q_8)$ and $\varphi(S, R) \leq 18$ (since $x_2^2x_1^{-2} \in
\overline{R}$ is of weight 6). Thus $\psi(Q_8; 2) \leq 18 = 6(2+1)$. Therefore, $\rho(Q_8) =2$ and hence
$\psi(Q_8) \leq 18$.

Now, let $\varphi(S, R) \leq 18$, where $S=\{x_1, x_2\}$ and $\langle S \, | \, R\rangle\in \mathcal{P}_2(Q_8)$.
Note that $B=\{x_1x_2, x_1^2, x_2^2, x_1^2x_2^{-1}, x_2^2x_1^{-1}$, $x_1x_2^{-1}x_1x_2^{-1}, x_2^2x_1$,
$x_1^3x_2^{-1}$, $x_2^2x_1^{-1}x_2x_1^{-1}$, $x_1^2x_2$, $x_2^3x_1^{-1}$, $x_1^3,x_2^3$,
$x_1^2x_2^{-1}x_1x_2^{-1}$, $x_1x_2x_1x_2^{-1}$, $x_2x_1x_2x_1^{-1},x_1x_2^{-1}x_1x_2^{-1}x_1x_2^{-1}$,
$x_1x_2x_1^{-1}x_2^{-1}$, $x_2^2x_1^{-2}\}$ is a set of pairwise independent elements of weight $4$ or $6$ in
$F(S)$. It is not difficult to see that $w\in F(S)$ and $4 \leq \lambda(w) \leq 6 $ imply $w$ is dependent with
an element of $B$. Let $B^{-1}=\{w^{-1}:w\in B\}$. Then $R \subseteq B\cup B^{-1}$. Clearly, the only possible
choices of $\{r_1^{\pm 1}, r_2^{\pm 1}\}$ are $ \{x_2^2x_1^{-2}, x_1x_2x_1x_2^{-1} \}$, $ \{x_2^2x_1^{-2},
x_2x_1x_2x_1^{-1} \}$ and $\{x_2x_1x_2x_1^{-1}$, $x_1x_2x_1x_2^{-1} \}$. Then $\lambda(r) \geq 6$ for $r\in R\cup
\overline{R}$. Thus, $\varphi(S,R) \geq 18$. Therefore, $\psi(Q_8)=18$. This proves parts (iv).

\smallskip

Clearly, $\mu(\mathbb{Z}^3)=3$. If $S_0=\{x_1, x_2, x_3\}$ and $R_0=\{x_ix_jx_i^{-1}x_j^{-1} \, : \, 1\leq i<
j\leq 3\}$ then $\langle S_0 \, | \, R_0\rangle\in \mathcal{P}_3(\mathbb{Z}^3)$ and $\varphi(S_0, R_0) \leq 24$
(since $x_1x_2^{-1}x_3x_1^{-1}x_2x_3^{-1}\in \overline{R}_0$ is of weight 6). Thus $\psi(\mathbb{Z}^3; 3) \leq 24
= 6(3+1)$. Therefore, $\rho(\mathbb{Z}^3) =3$ and hence $\psi(\mathbb{Z}^3) \leq 24$.

\medskip

\noindent {\em Claim.} If $w\in N(R_0)$ is not the identity then $\lambda(w) \geq 6$.

\smallskip

If $w\in N(R_0)$ is not the identity then clearly $\lambda(w) \neq 2$. Observe that, if $w \in F(S_0)$ with
$\lambda(w)=4$, then $w$ is dependent with an element of the set $C= \{x_i^2x_j^{-1}, x_ix_j^{-1}x_i x_j^{-1},
x_i^2,  x_ix_j$, $x_ix_j^{-1}x_i x_k^{-1} \, : \, ijk$ is a permutation of $123\}$. Since none of the element in
$C$ is in $N(R_0)$, it follows that $N(R_0)$ has no element of weight 4. This proves the claim.

Now, let $\varphi(S, R) \leq 24$, where $S=\{x_1, x_2, x_3\}$ and $\langle S \, | \, R\rangle\in
\mathcal{P}_3(\mathbb{Z}^3)$. Then $N(R) = N(R_0)$ and hence, by the claim, weight of each element of $R$ is at
least 6. This implies $\varphi(S, R) \geq 24$ and hence $\varphi(S, R) = 24$. Therefore, $\psi(\mathbb{Z}^3)=24$.
This completes the proof.
\end{proof}

\noindent {\em Proof of Theorem} \ref{theorem:T1}. Let $G = \pi(M, x)$ for some $x\in M$. To prove the theorem,
it is sufficient to show that any crystallization of $M$ needs at least $\psi(M) = \psi(G)$ vertices.

Let $(\Gamma,\gamma)$ be a crystallization of $M$ with $m$ vertices and let $\{1,2,3,4\}$ be the color set. Then,
by Proposition \ref{prop:gagliardi79b}, we know that $G$ has a presentation with $g_{ij}-1$ generators and $\leq
g_{ij}-1$ relations. Therefore, by the definition of $\mu(G)$ (in (\ref{mu(G)})), $\mu(G) \leq g_{ij}-1$. Then,
by part (ii) of Proposition \ref{prop:gagliardi79a},
\begin{align} \label{eq:m}
m = & \, 2(g_{12}+g_{13}+g_{14}) -4 \geq 6(\mu(G) +1) -4 = 6\mu(G) +2.
\end{align}

From the definition of $\rho(G)$ (in (\ref{rho(G)})), $6(\rho(G) +1) \geq \psi(G; \rho(G))$. Therefore, $m >
6(\rho(G)+1)$ implies $m > \psi(G; \rho(G))$. Thus, the result follows from this and Eq. \eqref{eq:m}.

Now, assume that $m \leq 6(\rho(G)+1)$. Then, by part (ii) of Proposition \ref{prop:gagliardi79a}, $g_{12}+
g_{13}+g_{14} \leq 2 + 3(\rho(G)+1)$. This implies, $g_{1j} \leq \rho(G)+1$ for some $j \in \{2, 3, 4\}$. Assume,
without loss, that $g_{12}\leq \rho(G)+1$.

As in Subsection \ref{crystal}, let $G_1, \dots, G_{q+1}$ be the components of $\Gamma_{12}$ and $H_1, \dots,
H_{q+1}$ be the components of $\Gamma_{34}$, where $q +1 =g_{34} = g_{12}  \leq \rho(G)+1$. By Proposition
\ref{prop:gagliardi79b}, $G$ has a presentation of the form $\langle  x_1, x_2,\dots, x_q \, |\, r_1, r_2,\dots,
r_q \rangle$, where $x_k$ corresponds to $G_k$ and $r_k$ corresponds to $H_k$ as in Subsection \ref{crystal}. Let
$S= \{x_1, x_2,\dots, x_q\}$ and $R=\{r_1, \dots, r_q\}$.

For $1\leq i\leq q$, let $r_i = x_{i_1}^{\varepsilon_1} \cdots x_{i_m}^{\varepsilon_m}$, where $x_{i_1}, \dots,
x_{i_m} \in \{x_1, \dots, x_q\}$ and $\varepsilon_j = \pm 1$ for $1\leq j\leq m$, $(x_{i_{j+1}},
\varepsilon_{j+1}) \neq (x_{i_j}, -\varepsilon_j)$ for $1\leq j \leq m-1$ and $(x_{i_m}, \varepsilon_m) \neq
(x_{i_1}, -\varepsilon_1)$.

\medskip

\noindent {\em Claim.} For $1\leq i\leq q$, the length of the cycle $H_i$ is at least $\lambda(r_i)$.

\smallskip

Consider the word $\tilde{r}_i$ (in $F(\{x_1, \dots, x_q, x_{q+1}\}$) which is obtained from $r_i$ by the
following rules: if $\varepsilon_j=\varepsilon_{j+1}$ for $1\leq j \leq m-1$, then replace
$x_{i_j}^{\varepsilon_j}$ by $x_{i_j}^{\varepsilon_j} x_{q+1}^{-\varepsilon_j}$ in $r_i$ and if $\varepsilon_m
=\varepsilon_{1}$, then replace $x_{i_m}^{\varepsilon_m}$ by $x_{i_m}^{\varepsilon_m} x_{q+1}^{-\varepsilon_m}$
in $r_i$. Observe that $\tilde{r}_i$ is non empty (since $r_i$ is non empty) and the number of letters in
$\tilde{r}_i$ is same as $\lambda(r_i)$ (see \eqref{tildar} and \eqref{l(w)}). The claim follows from this.

Now, we want to find a relation $r_{q+1}$ such that any $q$ of the relations from the set $\{r_1, r_2, \dots,
r_q$, $r_{q+1}\}$ together with the generators $x_1, x_2, \dots, x_q$ give a presentation of $G$. This implies,
$r_{q+1} \in \overline{R}$. Choose $r_{q+1}\in  \overline{R}$, such that $\lambda(r_{q+1})$ is least possible.
The length of the cycle $H_{q+1}$ (in $\Gamma_{34}$) corresponding to $r_{q+1}$ is at least  $\lambda(r_{q+1})$.
So, $m=$ the number of vertices required to yield the presentation $\langle S \, |\, R \rangle$  $ \geq
\lambda(r_1) + \lambda(r_2)+\cdots + \lambda(r_{q+1})$. So, $m \geq  \lambda (r_1) + \lambda (r_2)+ \cdots
+\lambda (r_{q+1}) = \varphi(S,R) \geq \psi(G;\rho(G))$. Thus, $m \geq \max \{\psi(G; \rho(G)), 6a(G)+2\} =
\psi(G)$. This proves the theorem. \hfill $\Box$

\bigskip

\noindent {\em Proof of Corollary} \ref{cor:h2>=6m}. Let $f_i$ be the number of $i$-cells in $X$. So, $f_0=4$.
Therefore, $g_2(X) = f_1-16+10 = f_1-6=f_1-12+6=h_2(X)$. Since $|X|$ is a closed 3-manifold, we have
$0=f_0-f_1+f_2-f_3 = 4-f_1+2f_3-f_3$. Thus, $f_1 = f_3+4$ and hence $g_2(X) =h_2(X) = f_3-2$. Therefore, by
Theorem \ref{theorem:T1}, $g_2(X) =h_2(X) =f_3-2 \geq \Psi(M) -2\geq \psi(M)-2$.

From the definition of $\psi(G)$, $\psi(G) \geq 6\mu(G) +2 \geq 6m(G)+2$. Thus, $\psi(M) =  \psi(\pi(M, \ast))
\geq 6m(\pi(M, \ast))+2$. Since any presentation of $\pi(M,\ast)$ has at least $\beta_1(M; \mathbb{F})$
generators, it follows that $m(M) =  m(\pi(M, \ast)) \geq \beta_1(M; \mathbb{F})$. The corollary now follows.
\hfill $\Box$

\begin{remark} \label{remark:R1}
{\rm If a crystallization $(\Gamma, \gamma)$ yields a presentation $\langle S \, | \, R\rangle$ then, from the
proof of Theorem \ref{theorem:T1}, we get $\varphi(S, R) \leq$ the number of vertices of $\Gamma$. }
\end{remark}

\begin{remark} \label{remark:nonunique}
{\rm  We found that $\rho(\mathbb{Z}^3) =3$ and $\varphi(S, R) =24$, where $\langle S \, | \, R\rangle \in
\mathcal{P}_3(\mathbb{Z}^3)$. On the other hand, if $S^{\,\prime}= \{x_1, \dots, x_5\}$ and $R^{\,\prime}=
\{x_1x_4^{-1}x_5x_3^{-1}, x_1x_5x_2^{-1}, x_3x_4x_2^{-1}, x_1x_3^{-1}x_5x_4^{-1}$, $x_5x_1x_2^{-1}\}$ then
$\langle S^{\,\prime} \, | \, R^{\,\prime} \rangle\in \mathcal{P}_5(\mathbb{Z}^3) \setminus
\mathcal{P}_4(\mathbb{Z}^3)$ with $\varphi(S^{\prime}, R^{\prime}) =24$. So, the minimum weight presentation of
$\mathbb{Z}^3$ is not unique. This is true for most of the groups. }
\end{remark}

\section{Uniqueness of some crystallizations}

Here, we are interested on crystallizations of 3-manifolds $M$ with $\psi(M)$ vertices. For seven
3-manifolds, we show that there exists a unique such crystallization for each of them.

Throughout this section and behind, $1, 2, 3, 4$ are the colors of a 4-colored graph $(\Gamma, \gamma)$ and
$g_{ij}$ is the number of components of $\Gamma_{ij}=\Gamma_{\{i, j\}}$ for $i\neq j$.

Let ${\mathcal X}$ be the pseudotriangulation of a connected closed 3-manifold $M$ determined by a
crystallization $(\Gamma, \gamma)$. So, $(\Gamma, \gamma)$ is contracted, i.e., $\Gamma_{\{i, j,k\}}$ 
is connected for $i, j, k$ distinct. 
For $1\leq i\leq 4$, we denote the vertex of ${\mathcal X}$ corresponding to
the color $i$ by $v_i$. We identify a vertex $u$ of $\Gamma$ with the corresponding facet $\sigma_u$ of
${\mathcal X}$. For a facet $u$ ($\equiv \sigma_u$) of ${\mathcal X}$, the $2$-face of $u$ not containing the
vertex $v_i$ will be denoted by $u_i$. Similarly, the edge of $u$ not containing the vertices $v_i, v_j$ will be
denoted by $u_{ij}$. Clearly, if $C_{2k}(u^1, u^2,  \dots, u^{2k})$ is a $2k$-cycle in $\Gamma$ with colors $i$
and $j$ alternately, then $u^1_{ij} = u^2_{ij} = \cdots = u^{2k}_{ij}$ in ${\mathcal X}$.

\begin{lemma}\label{lemma:no2cycle}
Let $\Gamma$ be a crystallization of a connected closed $3$-manifold $M$ with $m$ vertices. If  $\Gamma$ has a
$2$-cycle, then  either $M$ has a crystallization with $m-2$ vertices or $\pi_1(M, x)$ $($for $x\in M)$ is
isomorphic to the free product $\mathbb{Z} \ast H$ for some group $H$.
\end{lemma}

\begin{figure}[ht]
\tikzstyle{vert}=[circle, draw, fill=black!100, inner sep=0pt, minimum width=4pt] \tikzstyle{vertex}=[circle,
draw, fill=black!00, inner sep=0pt, minimum width=4pt] \tikzstyle{ver}=[] \tikzstyle{extra}=[circle, draw,
fill=black!50, inner sep=0pt, minimum width=2pt] \tikzstyle{edge} = [draw,thick,-] \tikzstyle{arrow} =
[draw,thick,->] \centering
\begin{tikzpicture}[scale=0.3]
\begin{scope}[]
\node[vertex] (1) at (-2,6){}; \node[vert] (2) at (-4,4){}; \node[vertex] (3) at (-6,2){}; \node[ver] (98) at
(-6,2.6){$v$}; \node[vert] (4) at (-8,2){}; \node[vertex] (5) at (-10,4){}; \node[vert] (6) at (-12,6){};
\node[vert] (7) at (2,6){}; \node[vertex] (8) at (4,4){}; \node[vert] (9) at (6,2){}; \node[ver] (10) at (8,2){};
\node[vert] (11) at (10,4){}; \node[vertex] (12) at (12,6){}; \node[ver] (100) at (-7,-7){$\Gamma$}; \node[ver]
(101) at (7,-7){$\Gamma^1$}; \node[extra] (25) at (-10,-6){}; \node[extra] (26) at (-7,-6){}; \node[extra] (27)
at (-4,-6){}; \node[extra] (28) at (-10,6){}; \node[extra] (29) at (-7,6){}; \node[extra] (30) at (-4,6){};
\node[extra] (31) at (10,-6){}; \node[extra] (32) at (7,-6){}; \node[extra] (33) at (4,-6){}; \node[extra] (34)
at (10,6){}; \node[extra] (35) at (7,6){}; \node[extra] (36) at (4,6){};

\node[ver] (a1) at (12,2){$1$}; \node[ver] (a2) at (12,1){$2$}; \node[ver](a3) at (12,0){$3$}; \node[ver](a4) at
(12,-1){$4$};

\node[ver] (a5) at (16,2){}; \node[ver](a6) at (16,1){}; \node[ver](a7) at (16,0){}; \node[ver] (a8) at
(16,-1){};
\end{scope}
\begin{scope}[rotate=180]
\node[vertex] (13) at (-2,6){}; \node[vert] (14) at (-4,4){}; \node[vertex] (15) at (-6,2){}; \node[ver] (16) at
(-8,2){}; \node[vertex] (17) at (-10,4){}; \node[vert] (18) at (-12,6){}; \node[vert] (19) at (2,6){};
\node[vertex] (20) at (4,4){}; \node[vert] (21) at (6,2){}; \node[ver] (99) at (6,2.6){$w$}; \node[vertex] (22)
at (8,2){}; \node[vert] (23) at (10,4){}; \node[vertex] (24) at (12,6){};
\end{scope}
\foreach \x/\y in {2/3,4/5,8/9,14/15,20/21,22/23}{
    \path[edge, dotted] (\x) -- (\y);}
    \foreach \x/\y in {1/2,3/4,5/6,7/8,11/12,13/14,17/18,19/20,21/22,23/24}{
    \path[edge, dashed] (\x) -- (\y);}
    \draw[edge] plot [smooth,tension=1.5] coordinates{(3) (-5,0) (21) };
    \draw[edge] plot [smooth,tension=1.5] coordinates{(3) (-7,0) (21) };
    \draw[line width=3pt, line cap=round, dash pattern=on 0pt off 2 \pgflinewidth] plot [smooth,tension=1.5]
    coordinates{(3) (-7,0) (21) };
    \draw[edge, dashed] plot [smooth,tension=1.5] coordinates{(9) (7,0) (15) };
    \draw[edge, dotted] plot [smooth,tension=0.5] coordinates{(11) (10) (16) (17) };
    \path[arrow] (-1,0) -- (1,0);

    \foreach \x/\y in {a1/a5,a2/a6}{
    \path[edge] (\x) -- (\y);}
    \draw[line width=3pt, line cap=round, dash pattern=on 0pt off 2 \pgflinewidth] (a1) -- (a5);
    \path[edge,dashed] (a4) -- (a8);
    \path[edge,dotted] (a3) -- (a7);
 \end{tikzpicture}
\caption{Cancellation of a dipole of type 2}\label{fig:dipole}
\end{figure}
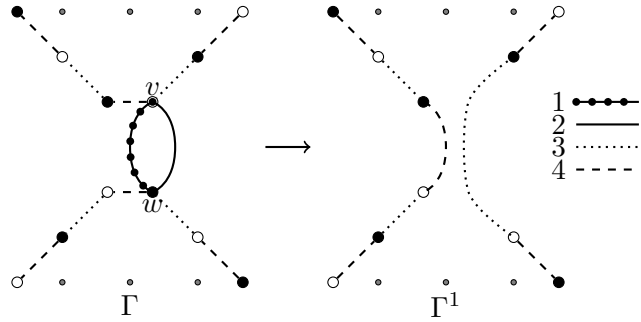

\vspace{-4mm}

\begin{proof}
Without loss, assume that $\Gamma$ has a $2$-cycle with color $1$ and $2$, i.e., $\Gamma_{12}$ has a component of
length $2$. If this $2$-cycle touches two different components of $\Gamma_{34}$ (say, at vertices $v$ and $w$,
respectively), then $\Gamma[\{v, w\}]$ is a  $3$-dimensional dipole of type $2$. Therefore, the crystallization
$\Gamma$ can be reduced to a  crystallization $\Gamma^1$ of $M$ with vertex set $V(\Gamma)\setminus\{v, w\}$ so
that $\Gamma^1_{12}$ (resp., $\Gamma^1_{34}$) has one less components than $\Gamma_{12}$ (resp., $\Gamma_{34}$)
as in Fig.  \ref{fig:dipole} (see \cite{fg82}). Thus, $M$ has a crystallization (namely, $\Gamma^1$) with $m-2$
vertices.

So, assume that the  $2$-cycle (say $G_1$) touches only one component (say, $H_1$) of $\Gamma_{34}$. Let $G_1,
\dots, G_{q+1}$ be the components of $\Gamma_{12}$ and $H_1, \dots, H_{q+1}$ be the components of $\Gamma_{34}$,
where $q+1 = g_{12} = g_{34}$. Let $x_1, \dots, x_{q+1}$ and $r_1, \dots, r_{q+1}$ be as in Proposition
\ref{prop:gagliardi79b}. Then, by Proposition \ref{prop:gagliardi79b}, $\pi_1(M, x)$ has a presentation of the
form $\langle x_1, x_2, \dots, x_q \, | \, r_2, r_3,\dots,r_{q+1} \rangle$. Since $G_1$ touches only $H_1$, from
the definition of $\tilde{r}_k$ in Eq. \eqref{tildar}, $\tilde{r}_k$ does not contain $x_1^{\pm 1}$ for $k\neq
1$. Therefore, $\langle x_1, x_2, \dots, x_q \, | \, r_2, \dots, r_{q+1} \rangle = \langle x_1 \rangle\ast
\langle x_2, \dots, x_q \, | \, r_2, \dots,r_{q+1} \rangle$. This proves the lemma.
\end{proof}

\begin{lemma}\label{unique:J12}
There exist exactly three $8$-vertex crystallizations of non-simply connected, connected, closed $3$-manifolds.
Moreover, these three are crystallizations of $S^2 \times S^1$, $\TPSS$ and $\mathbb{RP}^3$ respectively.
\end{lemma}

\begin{proof}
Let $(\Gamma, \gamma)$ be an $8$-vertex crystallization of a non simply connected, connected, closed $3$-manifold
$M$. By Proposition \ref{prop:gagliardi79a}, $g_{12}+g_{13}+g_{14} = 8/2+2=6$ and $g_{ij}=g_{kl}$ for $i, j, k,
l$ distinct. Since $\pi_1(M, *)$  has at least one generator, $g_{ij} \geq 2$ for $1\leq i \neq j\leq 4$. This
implies that $g_{ij}=2$ and hence $\Gamma_{ij}$ is of the form $C_2\sqcup C_6$ or $C_4\sqcup C_4$ for $1\leq i
\neq j\leq 4$.

\smallskip

\noindent {\em Case $1$:} $(\Gamma, \gamma)$ has a $2$-cycle. Since $M$ is not simply connected, $M$ has no
crystallization with less than 8 vertices. Therefore, by Lemma \ref{lemma:no2cycle}, $\pi_1(M,*)$ must have a
torsion free element. Again, $g_{ij}=2$ implies $\pi_1(M, *)$ is generated by one element and hence isomorphic to
$\mathbb{Z}$. Therefore, $M \cong S^{\hspace{.2mm}2}\times S^1$ or $\TPSS$. Assume, without loss, $\Gamma_{12} =
G_1\sqcup G_2$, where $G_1=C_2(v_3, v_4)$, $G_2=C_6(v_1,v_2,v_5,v_6,v_7,v_8)$. Then there is no edge between $v_3$ and $v_4$ of color 3 or 4 and (see the proof of Lemma
\ref{lemma:no2cycle}), $G_1$ touches only one component of $\Gamma_{34}$. Let $\Gamma_{34} = H_1 \sqcup H_2$, where $G_1\cap H_1 = \emptyset$. Let $x$ and $y$ be the generators
corresponding to the components $G_1$ and $G_2$ respectively. If $H_2$ is a 4-cycle then $H_2$ represents
$xy^{-1}xy^{-1}$ by choosing some $v_1$, $i, j$ as in Eq. \eqref{tildar}. But $xy^{-1}xy^{-1}$ does not give
identity relation by deleting $x$ or $y$. Therefore, $H_2$ is a 6-cycle and hence $H_1$ is a 2-cycle.
Similarly, $G_2 \cap H_2 = \emptyset$. Since the number of edges between any pair of vertices is at most 2, we can assume that $H_1=C_2(v_1, v_6)$.  Assume, without loss, that there is an edge of color 4 between $v_2$ and $v_3$. Since $\Gamma_{24}$ has two components, this implies $\Gamma_{24} = 
C_4(v_4, v_3, v_2, v_5) \sqcup C_4(v_8, v_1, v_6, v_7)$. So, there exists an edge of color 4 between 
$v_4$ and $v_5$ (resp. $v_7$ and $v_8$). 
Since $H_2$ is a 6-cycle on the vertex set $\{v_1, \dots, v_8\} \setminus \{v_1, v_6\}$, 
this implies that $H_2=C_6(v_2, v_3, v_8,v_7,v_4, v_5)$ or
$C_6(v_2, v_3, v_7, v_8, v_4, v_5)$. In the first case, $(\Gamma, \gamma) = \mathcal{J}_1$ and in the second
case,  $(\Gamma, \gamma) = \mathcal{J}_2$ given in Fig. \ref{fig:J12} (a) and (b) respectively. In the first
case, $\Gamma$ is bipartite. Therefore, $M$ is orientable and hence equal to $S^2 \times S^1$. In the second
case, $\Gamma$ is not bipartite. Therefore, $M$ is non-orientable and hence equal to $\TPSS$.


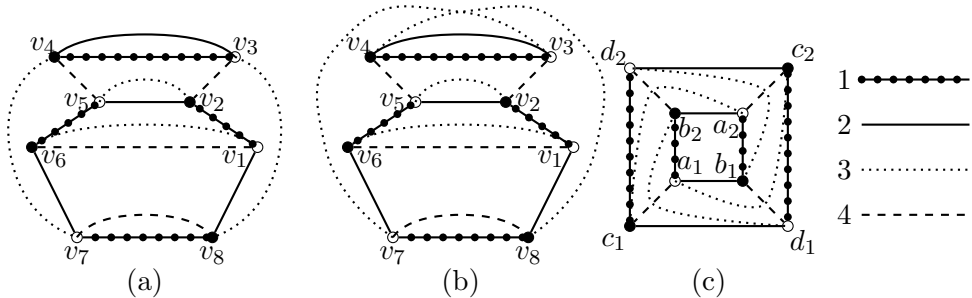
\begin{figure}[ht]
\tikzstyle{vert}=[circle, draw, fill=black!100, inner sep=0pt, minimum width=4pt]
\tikzstyle{vertex}=[circle, draw, fill=black!00, inner sep=0pt, minimum width=4pt]
\tikzstyle{ver}=[]
\tikzstyle{extra}=[circle, draw, fill=black!50, inner sep=0pt, minimum width=2pt]
\tikzstyle{edge} = [draw,thick,-]

\centering
\begin{tikzpicture}[scale=0.3]
\begin{scope}[shift={(-15,0)}]
\foreach \x/\y in {5/0,4/4,-2/2,-3/-4}{ \node[vertex] (\x) at (\x,\y){};} \foreach \x/\y in {2/2,-4/4,-5/0,3/-4}{
\node[vert] (\x) at (\x,\y){};} \foreach \x/\y in {5/3,5/2,2/-2,-2/-5,-5/-3,4/-4,3/-3}{ \path[edge] (\x) --
(\y);} \foreach \x/\y in {5/2,-2/-5,4/-4,3/-3}{ \draw [line width=3pt, line cap=round, dash pattern=on 0pt off
2\pgflinewidth]  (\x) -- (\y);} \draw[edge] plot [smooth,tension=1.5] coordinates{(-4)(0,5)(4)}; \foreach \x/\y
in {5/-5,4/2,-4/-2}{ \path[edge, dashed] (\x) -- (\y);} \draw[edge, dashed] plot [smooth,tension=1.5]
coordinates{(-3)(0,-3)(3)}; \draw[edge, dotted] plot [smooth,tension=1.5] coordinates{(-2)(0,3)(2)}; \draw[edge,
dotted] plot [smooth,tension=1.5] coordinates{(-5)(0,1)(5)}; \draw[edge, dotted] plot [smooth,tension=1.5]
coordinates{(-4)(-6,0)(-3)}; \draw[edge, dotted] plot [smooth,tension=1.5] coordinates{(4)(6,0)(3)}; \node[ver]
(100) at (4,-0.5){$v_1$}; \node[ver] (101) at (3,2){$v_2$}; \node[ver] (102) at (4.5,4.5){$v_3$}; \node[ver]
(103) at (-4.5,4.5){$v_4$}; \node[ver] (104) at (-3,2){$v_5$}; \node[ver] (105) at (-4,-0.5){$v_6$}; \node[ver]
(106) at (-3,-4.8){$v_7$}; \node[ver] (107) at (3,-4.8){$v_8$}; \node[ver] (108) at (0,-6){(a)};
\end{scope}
\begin{scope}[shift={(-1,0)}]
\foreach \x/\y in {5/0,4/4,-2/2,-3/-4}{ \node[vertex] (\x) at (\x,\y){};} \foreach \x/\y in {2/2,-4/4,-5/0,3/-4}{
\node[vert] (\x) at (\x,\y){};} \foreach \x/\y in {5/3,5/2,2/-2,-2/-5,-5/-3,4/-4,3/-3}{ \path[edge] (\x) --
(\y);} \foreach \x/\y in {5/2,-2/-5,4/-4,3/-3}{ \draw [line width=3pt, line cap=round, dash pattern=on 0pt off
2\pgflinewidth]  (\x) -- (\y);} \draw[edge] plot [smooth,tension=1.5] coordinates{(-4)(0,5)(4)}; \foreach \x/\y
in {5/-5,4/2,-4/-2}{ \path[edge, dashed] (\x) -- (\y);} \draw[edge, dashed] plot [smooth,tension=1.5]
coordinates{(-3)(0,-3)(3)}; \draw[edge, dotted] plot [smooth,tension=1.5] coordinates{(-2)(0,3)(2)}; \draw[edge,
dotted] plot [smooth,tension=1.5] coordinates{(-5)(0,1)(5)}; \draw[edge, dotted] plot [smooth,tension=1]
coordinates{(-3)(-6,0)(-5,5)(0,6)(4)}; \draw[edge, dotted] plot [smooth,tension=1]
coordinates{(3)(6,0)(5,5)(0,6)(-4)}; \node[ver] (200) at (4,-0.5){$v_1$}; \node[ver] (201) at (3,2){$v_2$};
\node[ver] (202) at (4.5,4.5){$v_3$}; \node[ver] (203) at (-4.5,4.5){$v_4$}; \node[ver] (204) at (-3,2){$v_5$};
\node[ver] (205) at (-4,-0.5){$v_6$}; \node[ver] (206) at (-3,-4.8){$v_7$}; \node[ver] (207) at (3,-4.8){$v_8$};
\node[ver] (208) at (0,-6){(b)};
\end{scope}

\begin{scope}[shift={(19,0)}]
\node[ver] (300) at (-3,3){$1$};
\node[ver] (301) at (-3,1){$2$};
\node[ver] (302) at (-3,-1){$3$};
\node[ver] (303) at (-3,-3){$4$};
\node[ver] (304) at (3,3){};
\node[ver] (305) at (3,1){};
\node[ver] (306) at (3,-1){};
\node[ver] (307) at (3,-3){};
\path[edge] (300) -- (304);
\draw [line width=3pt, line cap=round, dash pattern=on 0pt off 2\pgflinewidth]  (300) -- (304);
\path[edge] (301) -- (305);
\path[edge, dotted] (302) -- (306);
\path[edge, dashed] (303) -- (307);
\end{scope}
\begin{scope}[shift={(10,0)}]
\node[vertex] (a1) at (-1.5,-1.5){}; \node[vert] (a4) at (-1.5,1.5){}; \node[vert] (a2) at (1.5,-1.5){};
\node[vertex] (a3) at (1.5,1.5){}; \node[vert] (b1) at (-3.5,-3.5){}; \node[vertex] (b4) at (-3.5,3.5){};
\node[vertex] (b2) at (3.5,-3.5){}; \node[vert] (b3) at (3.5,3.5){}; \node[ver] (208) at (0,-6){(c)};

\foreach \x/\y in {a1/a2,a2/a3,a3/a4,a4/a1,b1/b2,b2/b3,b3/b4,b4/b1}{ \path[edge] (\x) -- (\y);} \foreach \x/\y in
{a3/a2,a1/a4,b3/b2,b1/b4}{ \draw [line width=3pt, line cap=round, dash pattern=on 0pt off 2\pgflinewidth]   (\x)
-- (\y);} \foreach \x/\y in {a1/b1,a2/b2,a3/b3,a4/b4}{ \path[edge, dashed] (\x) -- (\y);} \draw[edge, dotted]
plot [smooth,tension=0.5] coordinates{(a1)(2.5,-2.5)(b3)}; \draw[edge, dotted] plot [smooth,tension=0.5]
coordinates{(a2)(2.5,2.5)(b4)}; \draw[edge, dotted] plot [smooth,tension=0.5] coordinates{(a3)(-2.5,2.5)(b1)};
\draw[edge, dotted] plot [smooth,tension=0.5] coordinates{(a4)(-2.5,-2.5)(b2)}; \node[ver] () at
(-0.8,-0.8){$a_1$}; \node[ver] () at (-0.8,0.8){$b_2$}; \node[ver] () at (0.8,-0.8){$b_1$}; \node[ver] () at
(0.8,0.8){$a_2$}; \node[ver] () at (-4.2,-4.2){$c_1$}; \node[ver] () at (-4.2,4.2){$d_2$}; \node[ver] () at
(4.2,-4.2){$d_1$}; \node[ver] () at (4.2,4.2){$c_2$};
\end{scope}

\end{tikzpicture}
\vspace{-3mm} \caption{Crystallizations $\mathcal{J}_1$, $\mathcal{J}_2$ and $\mathcal{K}_{2,1}$} \label{fig:J12}
\end{figure}


\noindent {\em Case $2$:} $(\Gamma, \gamma)$ has no $2$-cycle. So, $\Gamma$ is a simple graph. Then,
$\Gamma_{ij}= C_4\sqcup C_4$ for $1\leq i \neq j\leq4$. Let $G_1=C_4(a_1,
b_1, a_2, b_2)$ and $G_2=
C_4(c_1, d_1, c_2, d_2)$ be the components of $\Gamma_{12}$. Assume, without loss, $a_1c_1$ is an edge of color $4$.
Then $a_2c_2$, $b_1d_1$, $b_2d_2$ are edges of color $4$ (since $\Gamma_{i4} = C_4\sqcup C_4$ for $1\leq i \leq
2$). If $a_1d_1$ is an edge of color $3$, then $C_4(b_1, a_1, d_1, c_1)$ would be a component of $\Gamma_{23}$. This implies $\Gamma[\{a_1, b_1, c_1, d_1\}]$ would be proper component of $\Gamma_{\{2, 3, 4\}}$. This is not possible since $(\Gamma, \gamma)$ is a contracted graph. Thus, $a_1d_1$ is not an edge of color $3$. Similarly, $a_1d_2$ is not an edge of color $3$.
These imply $a_1c_2$ is an edge of color $3$. Similarly, $b_1d_2$, $a_2c_1$ and 
$b_2d_1$ are edges of color $3$. Then,
$(\Gamma, \gamma) = \mathcal{K}_{2,1}$ given in Fig. \ref{fig:J12} (c).
Since $G_1 =C_4(a_1, b_1, a_2, b_2)$ and
$H_1=C_4(d_1, b_2, d_2, b_1)$ is a component of $\Gamma_{34}$, $\pi(M, \ast) = \langle x \, | \, x^2 \rangle\cong
\mathbb{Z}_2$. This implies that $M=\mathbb{RP}^3$. This completes the proof.
\end{proof}

\begin{lemma}\label{unique:K31}
There exists a unique $12$-vertex crystallization of $L(3,1)$.
\end{lemma}

\begin{proof}
By Lemma \ref{lemma:psi(G)} and Theorem \ref{theorem:T1}, $L(3,1)$ has no crystallization with less than 12 vertices. Let $(\Gamma, \gamma)$ be a $12$-vertex crystallization of $L(3,1)$. Since $\pi_1(L(3,1), *)$ ($\cong\mathbb{Z}_3$) has no torsion free element, by Lemma \ref{lemma:no2cycle}, $(\Gamma, \gamma)$ has no 2-cycle. So, $\Gamma$ is a simple graph and hence $g_{ij} \leq 3$ for $i\neq j$. 
Also (since $\cong\mathbb{Z}_3$ has at least one generator) $g_{ij} \geq 2$.
By Proposition \ref{prop:gagliardi79a},
$g_{12}+g_{13}+g_{14}=12/2+2=8$ and $g_{ij}=g_{kl}$ for $i, j, k, l$ distinct. So,  without loss, we can assume that 
$g_{12}=2, g_{13}=g_{14}=3$. Then $\Gamma_{ij} = C_4\sqcup C_4\sqcup C_4$ for $1\leq i\leq 2$ and $3\leq j \leq
4$. Let $G_1$, $G_2$ be the components  of $\Gamma_{12}$ and $H_1, H_2$ be the components of $\Gamma_{34}$ such
that $x_1, x_2$ represent the generators corresponding to $G_1$, $G_2$ respectively. Since $\langle x_j \,|\,
x_j^3\rangle$ is the only presentation in $\mathcal{P}_1(\mathbb{Z}_3)$, $H_i$ must yield the relations $x_j^{\pm
3}$, for $1 \leq i, j \leq 2$. Therefore, $G_i$ and $H_i$ are 6-cycles. Let $G_1=C_6(a_1,b_1,\dots, a_3, b_3)$ and $G_2=C_6(c_1,d_1, \dots, c_3, d_3)$. Assume, without loss, $a_1c_1\in \gamma^{-1}(4)$. 
Then $C_4(b_3, a_1, c_1, d_3) \subseteq \Gamma_{14}$ and hence $b_3d_3\in \gamma^{-1}(4)$. Similarly, $a_3c_3, b_2d_2, a_2c_2, b_1d_1\in \gamma^{-1}(4)$. Now, $a_1d_1\in \gamma^{-1}(3)$ 
$\Longrightarrow$ $C_4(a_1, d_1, c_1, b_1) \subseteq \Gamma_{23}$ $\Longrightarrow$ $\Gamma[\{a_1, b_1, c_1, d_1\}]$ is a component of $\Gamma_{\{2,3, 4\}}$. This is not possible since $\Gamma$ is a contracted graph. So, $a_1d_1\not\in \gamma^{-1}(3)$. Similarly, $a_1d_3\not\in \gamma^{-1}(3)$. 
Again, $a_1d_2\in \gamma^{-1}(3)$ $\Longrightarrow$ $C_4(a_1, d_2, c_2, b_1) \subseteq \Gamma_{23}$ 
$\Longrightarrow$ $C_4(a_2, b_1, c_2, d_1) \subseteq \Gamma_{13}$ $\Longrightarrow$ $\Gamma[\{a_2, b_1, c_2, d_1\}]$ is a component of $\Gamma_{\{1,3, 4\}}$, a contradiction. So, $a_1d_2\not\in \gamma^{-1}(3)$.
Therefore, up to an isomorphism, $a_1c_2 \in \gamma^{-1}(3)$. Then $b_1d_2$, $a_2c_3$, $b_2d_3$,
$a_3c_1$, $b_3d_1\in \gamma^{-1}(3)$ and hence $(\Gamma, \gamma) = \mathcal{K}_{3,1}$ given in Fig.
\ref{fig:K31} (a). Since $H_1=C_6(d_1, b_1, d_2, b_2, d_3, b_3)$ is one of
the two components of $\Gamma_{34}$, $(\Gamma, \gamma)$ yields $\langle x_1\, | \, x_1^3\rangle\cong \mathbb{Z}_3$.
So, $(\Gamma, \gamma)$ is a crystallization of $L(3,1)$. This completes the proof.
\end{proof}

\vspace{-6mm}

\begin{figure}[ht]
\tikzstyle{vert}=[circle, draw, fill=black!100, inner sep=0pt, minimum width=4pt]
\tikzstyle{vertex}=[circle, draw, fill=black!00, inner sep=0pt, minimum width=4pt]
\tikzstyle{ver}=[]
\tikzstyle{extra}=[circle, draw, fill=black!50, inner sep=0pt, minimum width=2pt]
\tikzstyle{edge} = [draw,thick,-]
\centering
\begin{tikzpicture}[scale=0.3]
\begin{scope}[shift={(-22,0)}]
\node[ver] (a1) at (8,-6){$1$};
\node[ver] (a2) at (8,-7){$2$};
\node[ver](a3) at (8,-8){$3$};
\node[ver](a4) at (8,-9){$4$};
\node[ver] (a5) at (13,-6){};
\node[ver](a6) at (13,-7){};
\node[ver](a7) at (13,-8){};
\node[ver] (a8) at (13,-9){};
\foreach \x/\y in {0/$a_1$,120/$a_2$,240/$a_3$}{
\node[ver] (\y) at (\x:4){\y};
    \node[vertex] (\y) at (\x:5){};
} \foreach \x/\y in
{60/$b_1$,180/$b_2$,300/$b_3$}{ \node[ver]
(\y) at (\x:4){\y};
    \node[vert] (\y) at (\x:5){};
} \foreach \x/\y in
{0/$c_1$,120/$c_2$,240/$c_3$}{
\node[ver] (\y) at (\x:9){\y};
    \node[vert] (\y) at (\x:8){};
} \foreach \x/\y in
{60/$d_1$,180/$d_2$,300/$d_3$}{
\node[ver] (\y) at (\x:9){\y};
    \node[vertex] (\y) at (\x:8){};
} \foreach \x/\y in {0/1,60/2,120/3,180/4,240/5,300/6}{ \node[ver] (\y) at (\x:6.5){}; }

\node[ver] (208) at (0,-9){(a)}; \node[ver] (208) at (22,-9){(b)};

\foreach \x/\y in {a1/a5,a2/a6,$a_1$/$b_1$,$b_1$/$a_2$,$a_2$/$b_2$,$b_2$/$a_3$,$a_3$/$b_3$,$a_1$/$b_3$,
$c_1$/$d_1$,$d_1$/$c_2$,$c_2$/$d_2$,$d_2$/$c_3$,$c_3$/$d_3$,$c_1$/$d_3$}{
    \path[edge] (\x) -- (\y);
}
\foreach \x/\y in {a1/a5,$a_2$/$b_1$,$a_3$/$b_2$,$a_1$/$b_3$,$c_2$/$d_1$,$c_3$/$d_2$,$c_1$/$d_3$}{
    \draw [line width=3pt, line cap=round, dash pattern=on 0pt off 2\pgflinewidth]  (\x) -- (\y);
}

\foreach \x/\y in
{$a_1$/$c_1$,$b_1$/$d_1$,$a_2$/$c_2$,$b_2$/$d_2$,$a_3$/$c_3$,$b_3$/$d_3$}{
    \path[edge, dashed] (\x) -- (\y);}
\foreach \x/\z/\y in
{$a_1$/2/$c_2$,$b_1$/3/$d_2$,$a_2$/4/$c_3$,$b_2$/5/$d_3$,$a_3$/6/$c_1$,
$b_3$/1/$d_1$}{ \draw[edge, dotted] plot [smooth,tension=0.5]
coordinates{(\x)(\z)(\y)}; }

\path[edge, dashed] (a4) --(a8);
\path[edge, dotted] (a3) -- (a7);
\end{scope}
\begin{scope}[shift={(-2.5,0)}, rotate=-120]
\foreach \x/\y in {300/x^{2},180/x^{4},60/x^{6}}{
\node[ver] (\y) at (\x:3.2){$\y$};
    \node[vertex] (\y) at (\x:4){};
}
\foreach \x/\y in {240/x^{3},120/x^{5},0/x^{1}}{
\node[ver] (\y) at (\x:3.1){$\y$};
    \node[vert] (\y) at (\x:4){};
}

\end{scope}
\begin{scope}[shift={(5,0)}]
\foreach \x/\y in {270/y^{4},90/y^{2}}{
\node[ver] (\y) at (\x:1){$\y$};
    \node[vertex] (\y) at (\x:2){};
}
\foreach \x/\y in {180/y^{1},0/y^{3}}{
\node[ver] (\y) at (\x:1){$\y$};
    \node[vert] (\y) at (\x:2){};
}

\end{scope}
\begin{scope}[rotate=-60]
\foreach \x/\y in {240/z^{3},120/z^{5},0/z^{1}}{
\node[ver] (\y) at (\x:9.8){$\y$};
    \node[vertex] (\y) at (\x:9){};
}
\foreach \x/\y in {300/z^{2},180/z^{4},60/z^{6}}{
\node[ver] (\y) at (\x:9.8){$\y$};
    \node[vert] (\y) at (\x:9){};
}

\end{scope}

\foreach \x/\y in
{x^{1}/x^{2},x^{2}/x^{3},x^{3}/x^{4},x^{4}/x^{5},x^{5}/x^{6},x^{6}/x^{1},y^{1}/y^{2},y^{2}/y^{3},y^{3}/y^{4},
y^{1}/y^{4},z^{1}/z^{2},z^{2}/z^{3},z^{3}/z^{4},z^{4}/z^{5},z^{5}/z^{6},z^{6}/z^{1}}{
    \path[edge] (\x) -- (\y);
}

\foreach \x/\y in
{x^{2}/x^{1},x^{4}/x^{3},x^{6}/x^{5},y^{2}/y^{1},
y^{3}/y^{4},z^{2}/z^{1},z^{4}/z^{3},z^{6}/z^{5}}{
    \draw [line width=3pt, line cap=round, dash pattern=on 0pt off 2\pgflinewidth]  (\x) -- (\y);
}

\foreach \x/\y in
{x^{2}/z^{2},x^{3}/z^{3},x^{4}/z^{4},x^{5}/y^{2},x^{6}/y^{1},x^{1}/z^{1},y^{3}/z^{5},y^{4}/z^{6}}{
    \path[edge, dashed] (\x) -- (\y);
}

\foreach \x/\y in
{x^{3}/z^{5},x^{2}/z^{4},x^{1}/z^{3},x^{6}/z^{2},z^{1}/y^{3},y^{2}/z^{6},y^{1}/x^{4},y^{4}/x^{5}}{
    \path[edge, dotted] (\x) -- (\y);
}
\end{tikzpicture}
\caption{Crystallizations $\mathcal{K}_{3,1}$ and $\mathcal{M}_{2,3}$}\label{fig:K31}
\end{figure}
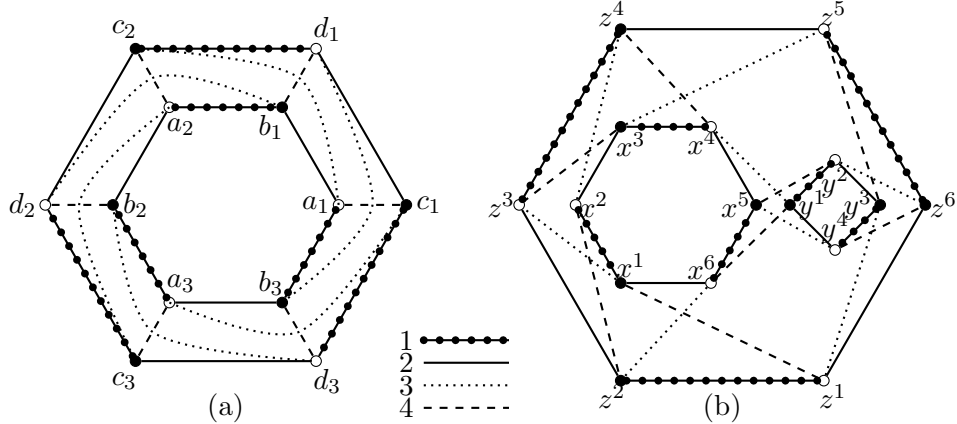

\vspace{-3mm}

\begin{lemma}\label{unique:M23}
There exists a unique $16$-vertex $4$-colored graph $(\Gamma, \gamma)$ which is a crystallization of a closed
connected $3$-manifold whose fundamental is $\mathbb{Z}_5$.
\end{lemma}

\begin{proof}

Let $(\Gamma, \gamma)$ be a $16$-vertex crystallization of a connected closed $3$-manifold $M$ and $\pi(M, \ast)
= \mathbb{Z}_5$.  Also, by Lemma \ref{lemma:psi(G)},
$\psi(M)=16$ and hence, by Theorem \ref{theorem:T1}, $(\Gamma, \gamma)$ is the crystallization of $M$ with
minimum number of vertices. Then, by Lemma \ref{lemma:no2cycle}, $(\Gamma, \gamma)$ has no 2-cycle. So, $\Gamma$ is a simple graph. Since $M$ is orientable, $\Gamma$ is bipartite.
By Proposition
\ref{prop:gagliardi79b} and Remark \ref{remark:R1}, $(\Gamma, \gamma)$ yields a presentation $\langle S\, | \,
R\rangle$ of $\mathbb{Z}_{5}$ with $\varphi(S,R)=16$.

\medskip

\noindent {\em Claim $1$.} If $\langle S=\{x_1, x_2\} \, | \, R=\{r_1, r_2\}\rangle\in
\mathcal{P}_2(\mathbb{Z}_5)$, $\varphi(S, R) =16$ and  $r_3 \in \overline{R}$ is of minimum weight then $\{r_1,
r_2, r_3\} =\{(x_1^3x_2^{-1})^{\pm 1}, (x_2^2x_1^{-1})^{\pm 1}, (x_1^2x_2)^{\pm 1} \}$ or
$\{(x_1^2x_2^{-1}x_1x_2^{-1})^{\pm 1}, (x_1x_2)^{\pm 1}$, $(x_2^2x_1^{-1}x_2x_1^{-1})^{\pm 1}\}$. (So, the set
$\{r_1, r_2, r_3\}$ has 16 choices.)

\smallskip

Let $B$ be the set as in the proof of Lemma \ref{lemma:psi(G)}. Then $w\in F(S)$ and $4
\leq \lambda(w) \leq 6 $ imply $w$ is dependent with an element of $B$. Since $\Gamma$ has no $2$-cycle,  $R$ has no element of weight less than 4. Since
$\varphi(S,R)=16$, we can assume that $4 \leq \lambda(r_1), \lambda(r_2) \leq 6$. Since $\langle S \,| \, R
\rangle \in \mathcal{P}_2(\mathbb{Z}_5)\setminus \mathcal{P}_1(\mathbb{Z}_5)$, the only possible choices of
$\{r_1^{\pm 1}, r_2^{\pm 1}\}$ are $\{x_1^3x_2^{-1}, x_2^2x_1^{-1}\}$, $\{x_2^2x_1^{-1}, x_1^2x_2\}$,
$\{x_1^3x_2^{-1}, x_1^2x_2\}$, $\{x_1^2x_2^{-1}x_1x_2^{-1}, x_1x_2\}$ or $\{x_1^2x_2^{-1}x_1x_2^{-1},
x_2^2x_1^{-1}x_2x_1^{-1}\}$. So, if $\langle S \, | \, R \rangle \in \mathcal{P}_2(\mathbb{Z}_{5})$ and
$\varphi(S,R)=16$, then $(r_1^{\pm 1}, r_2^{\pm 1}, r_3^{\pm 1})$ $= (x_1^3x_2^{-1}$, $x_2^2x_1^{-1}$,
$x_1^2x_2)$ or $(x_1^2x_2^{-1}x_1x_2^{-1}$, $x_1x_2$, $x_2^2x_1^{-1}x_2x_1^{-1})$. This proves  Claim 1.

If $g_{ij}=2$ for some $i\neq j$ then $(\Gamma, \gamma)$ yields a presentation $\langle S \, | \, R\rangle \in
\mathcal{P}_1(\mathbb{Z}_5)$ such that $\varphi(S, R) =16$ (see Remark \ref{remark:R1}), which is not possible by
Eq. \eqref{eq:4p}. Thus, $g_{ij} \geq 3$. Since (by Proposition \ref{prop:gagliardi79a}) $g_{12}+g_{13}+g_{14}
=16/2+2=10$, we can assume that $g_{12}=3=g_{13}, g_{14}=4$. In particular, if we choose generators (resp.,
relations) corresponding to the components of $\Gamma_{12}$ (resp., $\Gamma_{34})$ then $(\Gamma, \gamma)$ yields
a presentation $\langle S \, | \, R\rangle\in \mathcal{P}_2(\mathbb{Z}_5) \setminus \mathcal{P}_1(\mathbb{Z}_5)$
with $\varphi(S, R) =16$.

\medskip

\noindent {\em Claim $2$.} If $x_1, x_2$ are generators corresponding to two components of $\Gamma_{12}$ then the
relations corresponding to the components of $\Gamma_{34}$ are $(x_1^3x_2^{-1})^{\varepsilon_1}$,
$x_2^2x_1^{-1}$, $(x_1^2x_2)^{\varepsilon_2}$ for some $\varepsilon_1, \varepsilon_2 \in \{1, -1\}$.

\smallskip

Let $S, R, r_1, r_2, r_3$ be as in Claim 1. Then by choosing $(i, j) =(3, 4)$ or $(4, 3)$ as in Eq.
\eqref{tildar}, by Claim 1, we can assume $(r_1, r_2, r_3) = \big((x_1^3x_2^{-1})^{\pm 1}$, $x_2^2x_1^{-1}$,
$(x_1^2x_2)^{\pm 1} \big)$ or $\big((x_1^2x_2^{-1}x_1x_2^{-1})^{\pm 1}$, $(x_1x_2)^{-1}$, $(x_2^2x_1^{-1}x_2
x_1^{-1})^{\pm 1}\big)$. In the first case, Claim 2 trivially holds. In the second case, $ \tilde{r}_2=
(x_1x_3^{-1}x_2x_3^{-1})^{-1}$, where $x_3$ corresponds to the third component of $\Gamma_{12}$ (see Eq.
\eqref{tildar}). By deleting $x_2$ and renaming $x_3$ by $x_2$  in $\tilde{r}_2^{\pm 1}$,  we get the new
relation $x_2^2x_1^{-1}$. Claim 2 now follows from Claim 1.

To construct $\tilde{r}_i$ as in Eq. \eqref{tildar}, we can choose, without loss, $(i, j)=(4,3)$. Since
$g_{23}=g_{14}=4$, $\Gamma_{14}$ and $\Gamma_{23}$ are of the form $C_4\sqcup C_4\sqcup C_4\sqcup C_4$. Again,
$g_{12}=g_{34}= g_{24}=g_{13}=3$ implies $\Gamma_{13}$, $\Gamma_{24}$, $\Gamma_{12}$ and $\Gamma_{34}$ are of the
form $C_4\sqcup C_6\sqcup C_6$. Let $G_1, G_2, G_3$ be the components  of $ \Gamma_{12} $ and $H_1, H_2, H_3$ be
the components of $ \Gamma_{34} $ such that $x_1, x_2, x_3$ represent the generators corresponding to $G_1, G_2,
G_3$ respectively and $(x_1^3x_2^{-1})^{\varepsilon_1}$, $x_2^2x_1^{-1}$, $(x_1^2x_2)^{\varepsilon_2}$ represent
the relations corresponding to $H_1, H_2, H_3$  respectively.

Let $G_1 = C_{6}(x^1, \dots, x^{6})$, $G_2 = C_4 (y^1, \dots, y^4)$ and $G_3 = C_{6}(z^1, \dots, z^{6})$. Then to
form the relations $(x_1^3x_2^{-1})^{\varepsilon_1}$, $x_2^2x_1^{-1}$, $(x_1^2x_2)^{\varepsilon_2}$, we need to
add the following: (i) two edges of color $4$  between $G_1$ and $G_2$, (ii) four edges of colors $4$ between
$G_1$ and $G_3$, (iii) two edges of color $4$  between $G_2$ and $G_3$. 
These give all the 8 edges in $\gamma^{-1}(4)$. Therefore, we must have the following: (a) two 4-cycles between $G_1$ and $G_3$
in $\Gamma_{14}$, (b) one 4-cycle between $G_1$ and
$G_2$ in $\Gamma_{14}$, (c) one 4-cycle between $G_2$ and $G_3$ in $\Gamma_{14}$. 
So, the 4-cycle in $\Gamma_{24}$ is in between $G_1$ and $G_3$. Thus, up to an isomorphism,  
$\gamma^{-1}(4)$ is unique. In particular, we can assume that $\Gamma_{124}$ is as in Fig. \ref{fig:K31} (b). Now, $y^2x^{5}$ is an edge of color $4$ between $G_1$ and $G_2$. Thus,
$y^1$ (resp., $y^2$) is in $H_1$ or $H_2$. Assume, without loss, $y^1 \in H_1$. Then $y^2 \in H_2$. Since
$\Gamma$ is bipartite and $H_2$ represents $x_2x_3^{-1}x_2x_1^{-1}$, taking $v_1=x^{5}$ as in Eq. \eqref{tildar},
$H_2 = C_4(x^{5}, y^2, z^{6}, y^4)$. Since $\Gamma_{23} = C_4\sqcup C_4\sqcup C_4\sqcup C_4$ and
$\Gamma_{13}=C_4\sqcup C_6\sqcup C_6$, we have $x^4y^1$, $x^3z^5$, $x^2z^4$, $x^1z^3$, $y^3z^1$, 
$x^6z^2 \in \gamma^{-1}(3)$. Then $(\Gamma, \gamma)= \mathcal{M}_{2,3}$ of Fig. \ref{fig:K31} (b). This completes the 
proof.
\end{proof}

\begin{lemma}\label{unique:J3}
There exists a unique $18$-vertex crystallization of $S^3/Q_8$.
\end{lemma}

\begin{proof}
Let $(\Gamma, \gamma)$ be an $18$-vertex crystallization of $S^3/Q_8$.  By Lemma \ref{lemma:psi(G)} and  Theorem \ref{theorem:T1}, $(\Gamma, \gamma)$ is the crystallization 
of $S^3/Q_8$ with
minimum number of vertices. So, by Lemma \ref{lemma:no2cycle}, $(\Gamma, \gamma)$ has no 2-cycle. Thus, $\Gamma$ is a simple graph. Since $S^3/Q_8$ is orientable, $\Gamma$ is bipartite.
By Proposition
\ref{prop:gagliardi79b} and Remark \ref{remark:R1}, $(\Gamma, \gamma)$ yields a presentation $\langle S\, | \,
R\rangle$ of $Q_8$ with $\varphi(S,R)=18$.
Again, $(\Gamma,\gamma)$ has no 2-cycle implies $g_{ij} \leq 4$ for $i\neq j$. By
Proposition \ref{prop:gagliardi79a}, $g_{12}+g_{13}+g_{14}=18/2+2=11$. Assume, without loss, that $g_{12}=3$ and
$g_{13}=g_{14}=4$. Therefore, if we choose generators (resp., relations) correspond to the components of
$\Gamma_{12}$ (resp., $\Gamma_{34})$ then $(\Gamma, \gamma)$ yields a presentation $\langle S \, | \, R\rangle\in
\mathcal{P}_2(Q_8) \setminus \mathcal{P}_1(Q_8)$ with $\varphi(S, R) =18$. Then by the proof of part (v) in Lemma
\ref{lemma:psi(G)}, $R = \{(x_2^2x_1^{-2})^{\varepsilon_1}, (x_1x_2x_1x_2^{-1})^{\varepsilon_2},
(x_2x_1x_2x_1^{-1})^{\varepsilon_3}\}$ for some $\varepsilon_1, \varepsilon_2, \varepsilon_3 \in \{1, -1\}$.
Then, by choosing $(i, j) =(3, 4)$ or $(4, 3)$ as in Eq. \eqref{tildar}, we can assume that the three relations
correspond to components of $\Gamma_{34}$ are $(x_2^2x_1^{-2})^{\varepsilon_1},
(x_1x_2x_1x_2^{-1})^{\varepsilon_2}, x_2x_1x_2x_1^{-1}$, for some $\varepsilon_1, \varepsilon_2 \in \{-1, 1\}$.
 Since $\Gamma$ has no 2-cycle, $\Gamma_{ij} = C_4\sqcup C_4\sqcup
C_4\sqcup C_6$ for $1\leq i\leq 2$ and $3\leq j \leq 4$. Let $G_1, G_2, G_3$ be the components  of $ \Gamma_{12}
$ and $H_1, H_2, H_3$ be the components of $ \Gamma_{34} $ such that $x_1, x_2, x_3$ represent the generators
corresponding to $G_1, G_2, G_3$ and $(x_2^2x_1^{-2})^{\varepsilon_1}, x_2x_1x_2x_1^{-1},
(x_1x_2x_1x_2^{-1})^{\varepsilon_2}$ represent the relations corresponding to $H_1, H_2, H_3$ respectively. Then $G_i$, $H_i$
are 6-cycles for $1\leq i\leq 3$. Let $G_1 = C_6(a_1, \dots, a_6)$, $G_2 = C_6(b_1, \dots, b_6)$, $G_3 = C_6(c_1,
\dots, c_6)$.  Again, to form these relations, there are exactly three edges with color $4$ between $G_i$ and
$G_j$ for $i \neq j$. Since each of $\Gamma_{14}$ and $\Gamma_{24}$ has three 4-cycles, the three
edges with color $4$ between $G_i$ and $G_j$ for $i \neq j$, yield two $4$-cycles. Then, up to an
isomorphism, $\Gamma_{124}$ is as in Fig. \ref{fig:J3}. Same arguments hold for the color $3$.


\begin{figure}[ht]
\tikzstyle{vert}=[circle, draw, fill=black!100, inner sep=0pt, minimum width=4pt]
\tikzstyle{vertex}=[circle, draw, fill=black!00, inner sep=0pt, minimum width=4pt]
\tikzstyle{ver}=[]
\tikzstyle{extra}=[circle, draw, fill=black!50, inner sep=0pt, minimum width=2pt]
\tikzstyle{edge} = [draw,thick,-]
\centering
\begin{tikzpicture}[scale=0.3]
\begin{scope}[shift={(16,2)}, rotate=180]
\foreach \x/\y in {1/4,1/3,1/2,1/1}{
\node[ver] (d\y) at (\x,\y){};}
\foreach \x/\y in {4.5/4,4.5/3,4.5/2,4.5/1}{
\node[ver] (c\y) at (\x,\y){$\y$};}
\foreach \x/\y in {c1/d1,c2/d2}{
\path[edge] (\x) -- (\y);}
\path[edge, dashed] (c4) -- (d4);
\path[edge, dotted] (c3) -- (d3);
\draw [line width=3pt, line cap=round, dash pattern=on 0pt off 2\pgflinewidth]  (c1) -- (d1);
\end{scope}
\begin{scope}[]
\foreach \x/\y in {0/$c_1$,120/$c_3$,240/$c_5$}{
\node[ver] (\y) at (\x:2){\y};
\node[vertex] (\y) at (\x:3){};
}
\foreach \x/\y in {60/$c_2$,180/$c_4$,300/$c_6$}{
\node[ver] (\y) at (\x:2){\y};
\node[vert] (\y) at (\x:3){};
}
\node[ver] (25) at (0:0){$G_3$};
\end{scope}
\begin{scope}[shift={(7,7)}]
\foreach \x/\y in {0/$b_1$,120/$b_3$,240/$b_5$}{
\node[ver] (\y) at (\x:2){\y};
\node[vertex] (\y) at (\x:3){};
}
\foreach \x/\y in {60/$b_2$,180/$b_4$,300/$b_6$}{
\node[ver] (\y) at (\x:2){\y};
\node[vert] (\y) at (\x:3){};
}
\node[ver] (25) at (0:0){$G_2$};
\end{scope}
\begin{scope}[shift={(-7,7)}]
\foreach \x/\y in {0/$a_1$,120/$a_3$,240/$a_5$}{
\node[ver] (\y) at (\x:2){\y};
\node[vertex] (\y) at (\x:3){};
}
\foreach \x/\y in {60/$a_2$,180/$a_4$,300/$a_6$}{
\node[ver] (\y) at (\x:2){\y};
\node[vert] (\y) at (\x:3){};
}
\node[ver] (25) at (0:0){$G_1$};
\end{scope}

\foreach \x/\y in
{$a_1$/$a_2$,$a_2$/$a_3$,$a_3$/$a_4$,$a_4$/$a_5$,$a_5$/$a_6$,$a_6$/$a_1$,$b_1$/$b_2$,$b_2$/$b_3$,$b_3$/$b_4$,
$b_4$/$b_5$,$b_5$/$b_6$,$b_6$/$b_1$,$c_1$/$c_2$,$c_2$/$c_3$,$c_3$/$c_4$,$c_4$/$c_5$,$c_5$/$c_6$,$c_6$/$c_1$}{
\path[edge] (\x) -- (\y); } \foreach \x/\y in
{$a_1$/$a_2$,$a_3$/$a_4$,$a_5$/$a_6$,$b_1$/$b_2$,$b_3$/$b_4$,$b_5$/$b_6$,$c_1$/$c_2$, $c_3$/$c_4$,$c_5$/$c_6$}{
    \draw [line width=3pt, line cap=round, dash pattern=on 0pt off 2\pgflinewidth]  (\x) -- (\y);
}
\foreach \x/\y in {$a_1$/$b_4$,$b_5$/$c_2$,$a_6$/$c_3$,$a_5$/$c_4$,$b_6$/$c_1$,$a_2$/$b_3$}{
    \path[edge, dashed] (\x) -- (\y);
}

\draw[edge, dashed] plot[smooth,tension=1.5] coordinates{({$c_6$})(8,0)({$b_1$})};
\draw[edge, dashed] plot [smooth,tension=1.5] coordinates{({$a_4$})(-8,0)({$c_5$})};
\draw[edge, dashed] plot [smooth,tension=1.5] coordinates{({$a_3$})(0,11)({$b_2$})};
\draw[edge, dotted] plot [smooth,tension=1] coordinates{({$a_6$})(-6,-4)(3,-3)({$c_1$})};
\draw[edge, dotted] plot [smooth,tension=1] coordinates{({$a_2$})(-2.5,5)(-4.5,0)({$c_5$})};
\draw[edge, dotted] plot [smooth,tension=1] coordinates{({$a_1$})(-5,-3)({$c_6$})};
\draw[edge, dotted] plot [smooth,tension=1] coordinates{({$c_4$})(-3,2)(1,5)(7,2)({$b_1$})};
\draw[edge, dotted] plot [smooth,tension=1] coordinates{({$c_3$})(1,4)(8,1)(11,6)({$b_2$})};
\draw[edge, dotted] plot [smooth,tension=1] coordinates{({$c_2$})(0,5)({$b_3$})};
\draw[edge, dotted] plot [smooth,tension=1] coordinates{({$b_4$})(0,6)(-6,2)({$a_5$})};
\draw[edge, dotted] plot [smooth,tension=1] coordinates{({$a_3$})(-1,9)(5,3)({$b_6$})};
\draw[edge, dotted] plot [smooth,tension=1] coordinates{({$a_4$})(-8,11)(-1,10)({$b_5$})};
\end{tikzpicture}
\vspace{-5mm} \caption{Crystallization $\mathcal{J}_3$ of $S^3/\mathbb{Q}_8$}\label{fig:J3}
\end{figure}
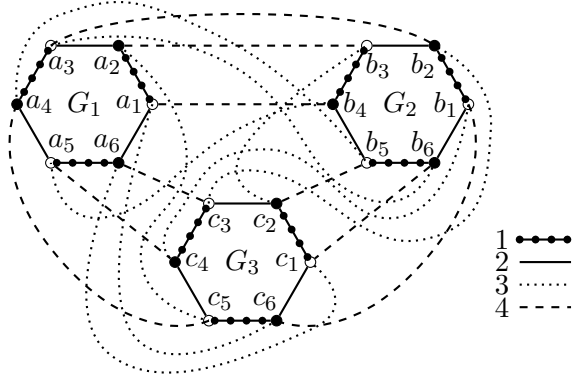


To construct $\tilde{r}_k$ as in Eq. \eqref{tildar}, choose $(i, j) = (4, 3)$. Since $H_2$ presents the relation
$x_2x_1x_2x_1^{-1}$, up to isomorphism, the starting vertex $v_1$ (as in Eq. \eqref{tildar}) is $a_2$ or $a_3$.
If $v_1=a_2$ then $H_2 = C_6(a_2, b_3, c_4, a_5, c_2, b_5)$ or $C_6(a_2, b_3, c_4, a_5, c_6, b_1)$. 
In the first
case, if $b_4c_3 \in \gamma^{-1}(3)$, then $b_4c_3$ lies in a cycle of size at least 8 in $\Gamma_{23}$,
which is not possible. Then the 4-cycle in $\Gamma_{13}$ between $G_2$ and $G_3$ must be $C_4(b_1, b_2, 
c_5, c_6)$. But this is not possible since $b_1c_6 \in \gamma^{-1}(4)$.  
In the second case, if $b_2c_5 \in \gamma^{-1}(3)$, then $b_2c_5$ lies in a cycle of size at least 8 
in $\Gamma_{13}$, which is not possible. Then the 4-cycle in $\Gamma_{23}$ between $G_2$ and $G_3$ 
must be $C_4(b_4, b_5, c_2, c_3)$. Ahain, this is not possible since $b_5c_2\in \gamma^{-1}(4)$. Thus, $v_1=a_3$. Now, if $b_2c_5$ is an edge of color $3$ then $a_4c_1$ and $a_3b_6$ must be edges of color $3$. Then
$b_5c_2$ must be an edge of color $3$ to make a 6-cycle in $\Gamma_{13}$, which is a contradiction (since
$b_5c_2$ is already  an edge of color $4$). Thus, $H_2=C_6(a_3, b_2, c_3, a_6, c_1, b_6)$. Since the three edges
with color $3$ between $G_2$ and $G_3$ yield two $4$-cycles (in $\Gamma_{13}$ and $\Gamma_{23}$), $b_1c_4$, $b_3c_2$ must be edges of color $3$
between $G_2$ and $G_3$. To make a 6-cycle in $\Gamma_{13}$, $a_5b_4$ must be an edge of color $3$. 
By similar arguments, $a_1c_6, a_2c_5, a_4b_5\in \gamma^{-1}(3)$. Then, $(\Gamma, \gamma) =
\mathcal{J}_3$ of Fig. \ref{fig:J3}.

Now, the components $H_1=C_6(a_2, b_3, c_2, b_5, a_4, c_5)$ and $H_3=C_6(b_4, a_1, c_6, b_1, c_4, a_5)$ of
$\Gamma_{34}$ yield the relations $x_2^2x_1^{-2}$ and $x_1x_2x_1x_2^{-1}$ respectively. Thus $(\Gamma, \gamma)$
yields the presentation $\langle x_1, x_2 \, | \, x_2^2x_1^{-2}, x_1x_2x_1x_2^{-1}\rangle \cong Q_8$. This completes the proof.
\end{proof}


\vspace{-10mm}

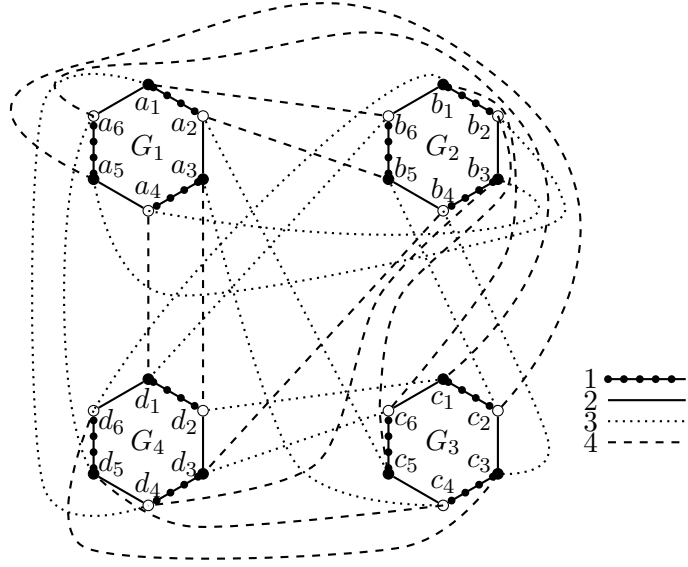
\begin{figure}[ht]
\tikzstyle{vert}=[circle, draw, fill=black!100, inner sep=0pt, minimum width=4pt] \tikzstyle{vertex}=[circle,
draw, fill=black!00, inner sep=0pt, minimum width=4pt] \tikzstyle{ver}=[] \tikzstyle{extra}=[circle, draw,
fill=black!50, inner sep=0pt, minimum width=2pt] \tikzstyle{edge} = [draw,thick,-] \centering
\begin{tikzpicture}[scale=0.28]
\begin{scope}[shift={(20,-3)}, rotate=180]
\foreach \x/\y in {1/4,1/3,1/2,1/1}{ \node[ver] (d\y) at (\x,\y){};} \foreach \x/\y in {6/4,6/3,6/2,6/1}{
\node[ver] (c\y) at (\x,\y){\y};} \foreach \x/\y in {c1/d1,c2/d2}{ \path[edge] (\x) -- (\y);} \path[edge, dashed]
(c4) -- (d4); \path[edge, dotted] (c3) -- (d3); \draw [line width=3pt, line cap=round, dash pattern=on 0pt off
2\pgflinewidth]  (c1) -- (d1);
\end{scope}
\begin{scope}[rotate=-90, shift={(-7,7)}]
\foreach \x/\y in {0/$b_4$,120/$b_2$,240/$b_6$}{ \node[ver] (\y) at (\x:2){\y}; \node[vertex] (\y) at (\x:3){};

} \foreach \x/\y in {60/$b_3$,180/$b_1$,300/$b_5$}{ \node[ver] (\y) at (\x:2){\y}; \node[vert] (\y) at (\x:3){};
} \node[ver] (25) at (0:0){$G_2$};
\end{scope}
\begin{scope}[rotate=-90, shift={(7,7)}]
\foreach \x/\y in {0/$c_4$,120/$c_2$,240/$c_6$}{ \node[ver] (\y) at (\x:2){\y}; \node[vertex] (\y) at (\x:3){}; }
\foreach \x/\y in {60/$c_3$,180/$c_1$,300/$c_5$}{ \node[ver] (\y) at (\x:2){\y}; \node[vert] (\y) at (\x:3){}; }
\node[ver] (25) at (0:0){$G_3$};
\end{scope}
\begin{scope}[rotate=-90, shift={(-7,-7)}]
\foreach \x/\y in {0/$a_4$,120/$a_2$,240/$a_6$}{ \node[ver] (\y) at (\x:2){\y}; \node[vertex] (\y) at (\x:3){}; }
\foreach \x/\y in {60/$a_3$,180/$a_1$,300/$a_5$}{ \node[ver] (\y) at (\x:2){\y}; \node[vert] (\y) at (\x:3){}; }
\node[ver] (25) at (0:0){$G_1$};
\end{scope}

\begin{scope}[rotate=-90, shift={(7,-7)}]
\foreach \x/\y in {0/$d_4$,120/$d_2$,240/$d_6$}{ \node[ver] (\y) at (\x:2){\y}; \node[vertex] (\y) at (\x:3){}; }
\foreach \x/\y in {60/$d_3$,180/$d_1$,300/$d_5$}{ \node[ver] (\y) at (\x:2){\y}; \node[vert] (\y) at (\x:3){}; }
\node[ver] (25) at (0:0){$G_4$};
\end{scope}

\foreach \x/\y in
{$a_1$/$a_2$,$a_2$/$a_3$,$a_3$/$a_4$,$a_4$/$a_5$,$a_5$/$a_6$,$a_6$/$a_1$,$b_1$/$b_2$,$b_2$/$b_3$,$b_3$/$b_4$,
$b_4$/$b_5$,$b_5$/$b_6$,$b_6$/$b_1$,$c_1$/$c_2$,$c_2$/$c_3$,$c_3$/$c_4$,$c_4$/$c_5$,$c_5$/$c_6$,$c_6$/$c_1$,
$d_1$/$d_2$,$d_2$/$d_3$,$d_3$/$d_4$,$d_4$/$d_5$,$d_5$/$d_6$,$d_6$/$d_1$}{ \path[edge] (\x) -- (\y); } \foreach
\x/\y in
{$a_1$/$a_2$,$a_3$/$a_4$,$a_5$/$a_6$,$b_1$/$b_2$,$b_3$/$b_4$,$b_5$/$b_6$,$c_1$/$c_2$,$c_3$/$c_4$,$c_5$/$c_6$,
$d_1$/$d_2$,$d_3$/$d_4$,$d_5$/$d_6$}{
    \draw [line width=3pt, line cap=round, dash pattern=on 0pt off 2\pgflinewidth]  (\x) -- (\y);
} \foreach \x/\y in {$a_1$/$b_6$,$a_2$/$b_5$,$a_3$/$d_2$,$a_4$/$d_1$,$b_4$/$d_3$}{
    \path[edge, dashed] (\x) -- (\y);
}

\draw[edge, dashed] plot [smooth,tension=1.5] coordinates{({$a_5$}) (-10,11) (10,10) ({$c_2$})}; \draw[edge,
dashed] plot [smooth,tension=1.5] coordinates{({$a_6$}) (-7,11) (10,9) ({$c_1$})}; \draw[edge, dashed] plot
[smooth,tension=0.5] coordinates{({$b_2$}) (10,5)(5,0) (4,-5) ({$c_5$})}; \draw[edge, dashed] plot
[smooth,tension=0.5] coordinates{({$b_1$}) (10,8.5) (10,2)  ({$c_6$})}; \draw[edge, dashed] plot
[smooth,tension=0.5] coordinates{({$b_3$})(4,0)  (1,-8) ({$d_4$})}; \draw[edge,dashed] plot [smooth,tension=0.5]
coordinates{({$c_4$})(-5,-11) ({$d_5$})}; \draw[edge, dashed] plot [smooth,tension=0.5] coordinates{({$c_3$})
(5,-12) (-10,-12) ({$d_6$})}; \foreach \x/\y in {$a_2$/$c_5$,$b_6$/$d_1$,$b_5$/$c_2$,$c_6$/$d_3$,$c_1$/$d_2$}{
    \path[edge, dotted] (\x) -- (\y);}
\draw[edge, dotted] plot [smooth,tension=1.5] coordinates{({$a_6$}) (-11,0) ({$d_5$})}; \draw[edge, dotted] plot
[smooth,tension=0.5] coordinates{({$a_1$}) (-12,9)(-12,-9) ({$d_4$})}; \draw[edge, dotted] plot
[smooth,tension=1.5] coordinates{({$b_1$}) (0,6) ({$d_6$})}; \draw[edge, dotted] plot [smooth,tension=0.5]
coordinates{({$a_3$}) (0,-8) ({$c_4$})}; \draw[edge, dotted] plot [smooth,tension=0.5] coordinates{({$b_4$})
(12,-7) ({$c_3$})}; \draw[edge, dotted] plot [smooth,tension=1.5] coordinates{({$a_4$}) (8,3) ({$b_3$})};
\draw[edge, dotted] plot [smooth,tension=0.5] coordinates{({$a_5$}) (-6,0) (11,3) ( 12.5,5) ({$b_2$})};
\end{tikzpicture}
\caption{Crystallization $\mathcal{J}_4$ of $S^1\times S^1\times S^1$}\label{fig:J4}
\end{figure}

\vspace{-4mm}


\begin{lemma} \label{unique:J4}
There exists a unique $24$-vertex crystallization of $S^1\times S^1\times S^1$.
\end{lemma}

\begin{proof}
Let $(\Gamma, \gamma)$ be a $24$-vertex crystallization of $(S^1)^3$.  By Lemma \ref{lemma:psi(G)} and   Theorem \ref{theorem:T1}, $(\Gamma, \gamma)$ is the crystallization 
of $(S^1)^3$ with
minimum number of vertices. So, by Lemma \ref{lemma:no2cycle}, $(\Gamma, \gamma)$ has no 2-cycle. Thus, $\Gamma$ is a simple graph. Since $(S^1)^3$ is orientable, $\Gamma$ is bipartite.
By Proposition
\ref{prop:gagliardi79b} and Remark \ref{remark:R1}, $(\Gamma, \gamma)$ yields a presentation $\langle S\, | \,
R\rangle$ of $\mathbb{Z}^3$ with $\varphi(S,R)=24$. Since any presentation of $\mathbb{Z}^3$ has at least three generators, $g_{ij} \geq 4$ for $i \neq j$. 
By Proposition \ref{prop:gagliardi79a}, $g_{12}+g_{13}+g_{14}=14$ and $g_{ij}=g_{kl}$ for $i, j, k, l$
distinct. 

\medskip

\noindent {\em Claim.} $(\Gamma, \gamma)$ does not yield a presentation $\langle S \, | \,R  \rangle \in
\mathcal{P}_4(\mathbb{Z}^3) \setminus \mathcal{P}_3(\mathbb{Z}^3)$ with $\varphi(S,R)=24$.

\smallskip

Assume $\langle S \, | \,R  \rangle \in \mathcal{P}_4(\mathbb{Z}^3) \setminus \mathcal{P}_3(\mathbb{Z}^3)$, where
$S=\{x_1, x_2, x_3, x_4\}$. Then $A:=\{(x_k^2x_l^{-1})^{\pm 1}$, $(x_jx_k^{-1}x_jx_l^{-1})^{\pm 1},
x_ix_j^{-1}x_kx_l^{-1}, (x_jx_kx_l^{-1})^{\pm 1}, (x_kx_l^{-1})^{2}, x_l^{\pm 2}$ $: \, i, j, k, l \mbox{ are
distinct}\}$ is the set of all relations of weight four. Since $\Gamma$ has no 2-cycle, $R$ has no
element of weight two. This implies that $R$ has at least three elements of weights four. Since $\mathbb{Z}^3$
has no torsion element, $x_l^{\pm 2}, (x_kx_l^{-1})^{2} \not\in R$. Consider an element $w\in R\cap A$. Assume,
without loss, $w = (w_1x_4^{-1})^{\pm 1}$. Then $\langle S_1 \, | \,R_1  \rangle \in
\mathcal{P}_3(\mathbb{Z}^3)$, where $S_1=\{x_1, x_2, x_3\}$ and $R_1$ consists of the elements $\bar{r}$, where
$\bar{r}$ can be obtained from a relation $r \in R$ by replacing $x_4$ by $w_1$. Let $A(w) := \{\bar{r} \, : \, r
\in A \setminus \{w^{\pm 1}\}\}$. Observe that the weights of the elements in $A(w)$ are 6 or 8.

Since $\langle S_1 \, | \,R_1  \rangle \in \mathcal{P}_3(\mathbb{Z}^3) \setminus \mathcal{P}_2(\mathbb{Z}^3)$, we
have $N(R_1) = N(R_0)$, where $R_0 = \{x_1x_2x_1^{-1}x_2^{-1}$, $x_1x_3x_1^{-1}x_3^{-1}$,
$x_2x_3x_2^{-1}x_3^{-1}\}$ and hence $N(R_1)$ has no element of weight less than 6 (see the proof of part (v) of
Lemma \ref{lemma:psi(G)}). Again, since $\#(R_1 \cap A(w)) \geq 2$, $R_1$ has at least two elements of weights 6
or 8. Observe that $D:=\{x_ix_jx_i^{-1}x_j^{-1}$, $x_ix_j^{-1}x_kx_i^{-1}x_jx_k^{-1}$,
$x_ix_jx_k^{-1}x_i^{-1}x_kx_j^{-1}$, $x_ix_jx_i^{-1}x_kx_j^{-1}x_k^{-1} \, : \{i,j,k\}=\{1,2,3\}\}$ is the set of
all relations of weights at most 8 in $N(R_0)$. So, $R_1$ has at least two independent relations from $D\cap
A(w)$. But $D\cap A(w)$ does not contain two such elements, a contradiction. This proves the claim.

By the claim, $g_{ij} \neq 5$ for all $1\leq i \neq j \leq 4$. So, we can assume that $g_{12}=g_{13}=4$ and
$g_{14}=6$. Then all the components of $\Gamma_{14}$ and $\Gamma_{23}$ are $4$-cycles. Let $\Gamma_{12} = G_1
\sqcup\cdots\sqcup G_4$ and $\Gamma_{34} = H_1\sqcup\cdots\sqcup H_4$ such that $x_1, \dots, x_4$ represent the
generators corresponding to $G_1, \dots, G_4$ respectively and $r_1, \dots, r_4$ represent the relations
corresponding to $H_1, \dots, H_4$ respectively. To construct $\tilde{r}_k$ as in Eq. \eqref{tildar}, choose $(i,
j) = (4, 3)$. Thus $(\Gamma, \gamma)$ yields a presentation $\langle S =\{x_1, x_2, x_3\} \, | \,R =\{r_1, r_2,
r_3\} \rangle \in \mathcal{P}_3(\mathbb{Z}^3) \setminus \mathcal{P}_2(\mathbb{Z}^3)$ with $\varphi(S,R)=24$. Then
$R$ contains three independent relations of weight 6 from the set $\{x_ix_jx_i^{-1}x_j^{-1}$,
$x_ix_j^{-1}x_kx_i^{-1}x_jx_k^{-1} \, : \, \{i, j, k\} =\{1, 2, 3\}\}$ (see the proof of part (v) of Lemma
\ref{lemma:psi(G)}). Without loss, we can assume that $R= \{x_1x_2x_1^{-1}x_2^{-1}$,
$(x_2x_3x_2^{-1}x_3^{-1})^{\varepsilon_1}$, $(x_1x_3x_1^{-1}x_3^{-1})^{\varepsilon_2}\}$ for some $\varepsilon_1,
\varepsilon_2 \in \{1, -1\}$. Then, all the components of $\Gamma_{12}$ and  $\Gamma_{34}$ are 6-cycles.
Similarly, all the components of $\Gamma_{13}$ and $\Gamma_{24}$ are 6-cycles. Let $G_{1}= C_6(a_1, \dots, a_6)$,
$G_{2}= C_6(b_1, \dots, b_6)$, $G_{3}= C_6(c_1, \dots, c_6)$ and $G_{4}= C_6(d_1, \dots, d_6)$. To form the
relations, there are exactly two edges of color $3$ (resp., $4$) between $G_i$ and $G_j$ for $1\leq i \neq j \leq
4$.  Then, up to an isomorphism, $\Gamma_{124}$ is as in Fig. \ref{fig:J4}. Now for the relation
$x_1x_2x_1^{-1}x_2^{-1}$, we can choose $v_1=b_6$ as in Eq. \eqref{tildar}. Then the cycle for
$x_1x_2x_1^{-1}x_2^{-1}$ is $H_1=C_6(b_6, a_1, d_4, b_3, a_4, d_1)$. Since $\Gamma_{23}$ consists of $4$-cycles,
it follows that $a_6d_5$, $a_5b_2$, $b_1d_6\in \gamma^{-1}(3)$. Then the cycle for the relation
$x_2x_3x_2^{-1}x_3^{-1}$ is $H_2=C_6(c_6, b_1, d_6, c_3, b_4, d_3)$. Again (since $\Gamma_{23}$ is union of
$4$-cycles and $\Gamma_{13}$ is union of $6$-cycles), $d_2c_1$, $b_5c_2$, $a_3c_4$, $a_2c_5\in \gamma^{-1}(3)$. Then $(\Gamma, \gamma) = \mathcal{J}_4$ of Fig. \ref{fig:J4}.

Now, the components $H_1$,
$H_2$ and $H_3 = C_6(c_1, a_6, d_5, c_4, a_3, d_2)$ yield the relations  $x_1x_2x_1^{-1}x_2^{-1}$, $x_2x_3
x_2^{-1} x_3^{-1}$ and $ x_1x_3x_1^{-1}x_3^{-1}$ respectively. Thus $(\Gamma, \gamma)$ yields the presentation $\langle  x_1, x_2, x_3\, | \, x_1x_2x_1^{-1}x_2^{-1}$, $x_2x_3x_2^{-1}x_3^{-1}, x_1x_3x_1^{-1}
x_3^{-1} \rangle \cong \mathbb{Z}^3$. This completes the proof.
\end{proof}

\begin{remark} \label{remark:MKN}
{\rm The crystallizations $\mathcal{K}_{2,1}$, $\mathcal{K}_{3,1}$ and $\mathcal{M}_{3,2}$ (in Figures
\ref{fig:J12} and \ref{fig:K31}) were originally found by Gagliardi {\em et al.} (\cite{fgg86, ga79b}). The first
two have the following natural generalization: Consider the bipartite graph $\Gamma$ consists of two disjoint
$2p$-cycles $G_1 = C_{2p}(a_1, b_1, \dots, a_p, b_p)$, $G_2= C_{2p}(c_1,d_1, \dots, c_p,d_p)$ together with $4p$
edges $a_ic_i$, $b_id_i$, $a_ic_{i+q}$, $b_id_{i+q}$ for $1\leq i\leq p$.  Consider the edge-coloring $\gamma$
with colors $1, 2, 3, 4$ of $\Gamma$ as: $\gamma(b_ia_{i+1}) = \gamma(d_ic_{i+1}) = 1$, $\gamma(a_ib_{i}) =
\gamma(c_id_{i}) = 2$, $\gamma(a_ic_{i+q}) = \gamma(b_id_{i+q}) = 3$ and $\gamma(a_ic_{i}) = \gamma(b_id_{i}) =
4$, $1\leq i\leq p$. (Summations in the subscripts are modulo $p$.) Then, $\mathcal{K}_{p,q} = (\Gamma, \gamma)$
is a $4p$-vertex crystallization of $L(p,q)$, for $p\geq 2$ and $q\geq 1$. This series is more or less known in
the literature. In the next section, we present some generalizations of $\mathcal{M}_{3,2}$. }
\end{remark}

\section{Two series of crystallizations of lens spaces}

Generalizing the construction of
$\mathcal{M}_{3,2}$ (Fig. \ref{fig:K31} (b)) we have obtained the following
series of crystallizations.

\begin{eg}[\boldmath{$4(k+q-1)$}-{\bf
vertex crystallization of \boldmath{$L(kq-1, q)$}}] \label{ex:Mkq} {\rm Let $q\geq 3$. For each $k \geq 2$, we
construct a $4(k+q-1)$-vertex 4-colored simple graph $\mathcal{M}_{k,q}=(\Gamma^k, \gamma^k)$ with the color set
$\{1, 2, 3, 4\}$ inductively which yields the presentation $\langle  x, y \, | \, x^q y^{-1}, y^kx^{-1} \rangle$.
For this, we want $g^k_{12}=g^k_{34}=3$. Then, without loss, $g^k_{13}=g^k_{24}=k+q-2$ and
$g^k_{14}=g^k_{23}=k+q-1$, where $g^k_{ij}$ is the number of components of $\Gamma^k_{ij}$ for $i \neq  j$.
These imply, $\Gamma^k_{14}$ and $\Gamma^k_{23}$ must be union of 4-cycles and  $\Gamma^k_{13}$ (resp.,
$\Gamma^k_{24}$) has two 6-cycles and $(k+q-4)$ \, $4$-cycles. Then, by Proposition \ref{prop:gagliardi79a},
$\mathcal{M}_{k,q}$ would be a crystallization of some connected closed 3-manifold $M_k$.


\begin{figure}[ht]
\tikzstyle{vert}=[circle, draw, fill=black!100, inner sep=0pt, minimum width=4pt] \tikzstyle{vertex}=[circle,
draw, fill=black!00, inner sep=0pt, minimum width=4pt] \tikzstyle{ver}=[] \tikzstyle{extra}=[circle, draw,
fill=black!50, inner sep=0pt, minimum width=2pt] \tikzstyle{edge} = [draw,thick,-] \centering
\begin{tikzpicture}[scale=0.5]
\begin{scope}[shift={(-2,0)}]
\foreach \x/\y in
{51.43/x^{2q-3},102.86/x^{2q-5},154.28/x^{7},205.72/x^{5},257.14/x^{3},308.57/x^{1},360/x^{2q-1}}{ \node[ver]
(\y) at (\x:3.2){$\y~~$};
    \node[vert] (\y) at (\x:4){};
} \foreach \x/\y in
{77.14/x^{2q-4},128.575/x^{8},180/x^{6},231.43/x^{4},282.86/x^{2},334.29/x^{2q},25.72/x^{2q-2}}{ \node[ver] (\y)
at (\x:3.1){$\y~~$};
    \node[vertex] (\y) at (\x:4){};
} \foreach \x/\y in {109.29/e1,115.72/e2,122.14/e3}{ \node[extra] (\y) at (\x:4){}; } \node[ver] (a^{7}) at
(115:7){}; \node[ver] (a^{8}) at (117:5){};
\end{scope}
\begin{scope}[shift={(5,0)}, rotate=-90]
\foreach \x/\y in {0/y^{4},180/y^{2}}{ \node[ver] (\y) at (\x:0.6){$\y$};
    \node[vertex] (\y) at (\x:1.3){};
} \foreach \x/\y in {90/y^{3},270/y^{1}}{ \node[ver] (\y) at (\x:0.6){$\y$};
    \node[vert] (\y) at (\x:1.3){};
} \node[ver] (a^{9}) at (240:2){}; \node[ver] (a^{10}) at (300:2){};
\end{scope}
\begin{scope}[]
\foreach \x/\y in
{51.43/z^{2q-3},102.86/z^{2q-5},154.28/z^{7},205.72/z^{5},257.14/z^{3},308.57/z^{1},360/z^{2q-1}}{ \node[ver]
(\y) at (\x:9.8){$~\y$};
    \node[vert] (\y) at (\x:9){};
} \foreach \x/\y in
{77.14/z^{2q-4},128.575/z^{8},180/z^{6},231.43/z^{4},282.86/z^{2},334.29/z^{2q},25.72/z^{2q-2}}{ \node[ver] (\y)
at (\x:9.8){$~\y$};
    \node[vertex] (\y) at (\x:9){};
}

\foreach \x/\y in {102.86/b^{2q-5},154.28/b^{7},205.72/b^{5},257.14/b^{3},308.57/b^{1}}{

    \node[ver] (\y) at (\x:7){};
} \foreach \x/\y in {128.575/b^{8},180/b^{6},231.43/b^{4},282.86/b^{2},334.29/b^{2q},360/b^{2q-1}}{

    \node[ver] (\y) at (\x:7){};
} \foreach \x/\y in {77.14/b^{2q-4},25.72/b^{2q-2},51.43/b^{2q-3}}{

    \node[ver] (\y) at (\x:6){};
} \foreach \x/\y in {109.29/e4,115.72/e5,122.14/e6}{ \node[extra] (\y) at (\x:9){}; } \node[ver] (a^{2q-5}) at
(115:8){}; \node[ver] (a^{2q-4}) at (121:6.3){}; \node[ver] (4) at (11,-6){$4$}; \node[ver] (3) at (11,-5){$3$};
\node[ver](2) at (11,-4){$2$}; \node[ver](1) at (11,-3){$1$}; \node[ver] (8) at (14,-6){}; \node[ver](7) at
(14,-5){}; \node[ver](6) at (14,-4){}; \node[ver] (5) at (14,-3){};
\end{scope}
\foreach \x/\y in
{x^{2}/x^{3},x^{4}/x^{5},x^{6}/x^{7},x^{2q-4}/x^{2q-3},x^{2q-2}/x^{2q-1},x^{2q}/x^{1},y^{3}/y^{2},
y^{1}/y^{4},1/5,2/6}{
    \path[edge] (\x) -- (\y);
} \foreach \x/\y in
{x^{2}/x^{1},x^{4}/x^{3},x^{6}/x^{5},x^{7}/x^{8},x^{2q-4}/x^{2q-5},x^{2q-2}/x^{2q-3},x^{2q}/x^{2q-1},
y^{3}/y^{4},y^{1}/y^{2}}{
    \path[edge] (\x) -- (\y);
} \foreach \x/\y in
{x^{2}/x^{1},x^{4}/x^{3},x^{6}/x^{5},x^{7}/x^{8},x^{2q-4}/x^{2q-5},x^{2q-2}/x^{2q-3},x^{2q}/x^{2q-1},
y^{3}/y^{4},y^{1}/y^{2},1/5}{
    \draw [line width=3pt, line cap=round, dash pattern=on 0pt off 2\pgflinewidth]  (\x) -- (\y);
} \foreach \x/\y in {z^{2}/z^{3},z^{4}/z^{5},z^{6}/z^{7},z^{2q-4}/z^{2q-3},z^{2q-2}/z^{2q-1},z^{2q}/z^{1}}{
    \path[edge] (\x) -- (\y);
} \foreach \x/\y in {z^{2}/z^{1},z^{4}/z^{3},z^{6}/z^{5},z^{7}/z^{8},z^{2q-4}/z^{2q-5},z^{2q-2}/z^{2q-3},
z^{2q}/z^{2q-1}}{
    \path[edge] (\x) -- (\y);
} \foreach \x/\y in {z^{2}/z^{1},z^{4}/z^{3},z^{6}/z^{5},z^{7}/z^{8},z^{2q-4}/z^{2q-5},z^{2q-2}/z^{2q-3},
z^{2q}/z^{2q-1}}{
    \draw [line width=3pt, line cap=round, dash pattern=on 0pt off 2\pgflinewidth]  (\x) -- (\y);
} \foreach \x/\y in
{x^{1}/z^{1},x^{2}/z^{2},x^{3}/z^{3},x^{4}/z^{4},x^{5}/z^{5},x^{6}/z^{6},x^{7}/z^{7},x^{8}/z^{8},
x^{2q-5}/z^{2q-5},x^{2q-4}/z^{2q-4},
x^{2q-3}/z^{2q-3},x^{2q-2}/z^{2q-2},x^{2q-1}/y^{2},x^{2q}/y^{1},y^{3}/z^{2q-1},y^{4}/z^{2q},4/8}{
    \path[edge, dashed] (\x) -- (\y);

} \foreach \x/\y in {x^{7}/a^{7},x^{8}/a^{8},z^{2q-5}/a^{2q-5},z^{2q-4}/a^{2q-4},3/7}{
    \path[edge,  dotted] (\x) -- (\y);

} \foreach \x/\z/\y in
{x^{2q}/b^{1}/z^{2},x^{1}/b^{2}/z^{3},x^{2}/b^{3}/z^{4},x^{3}/b^{4}/z^{5},x^{4}/b^{5}/z^{6},
x^{5}/b^{6}/z^{7},x^{6}/b^{7}/z^{8},
x^{2q-5}/b^{2q-4}/z^{2q-3},x^{2q-4}/b^{2q-3}/z^{2q-2},y^{2}/b^{2q-1}/z^{2q},x^{2q-3}/b^{2q-2}/z^{2q-1},
y^{3}/b^{2q}/z^{1},y^{1}/a^{9}/x^{2q-2},y^{4}/a^{10}/x^{2q-1}}{

\draw[edge, dotted] plot [smooth,tension=1.5] coordinates{(\x) (\z) (\y) }; }
\end{tikzpicture}
\vspace{-4mm} \caption{Crystallization $\mathcal{M}_{2,q}$ of L$(2q-1,q)$}\label{fig:M2q}
\end{figure}
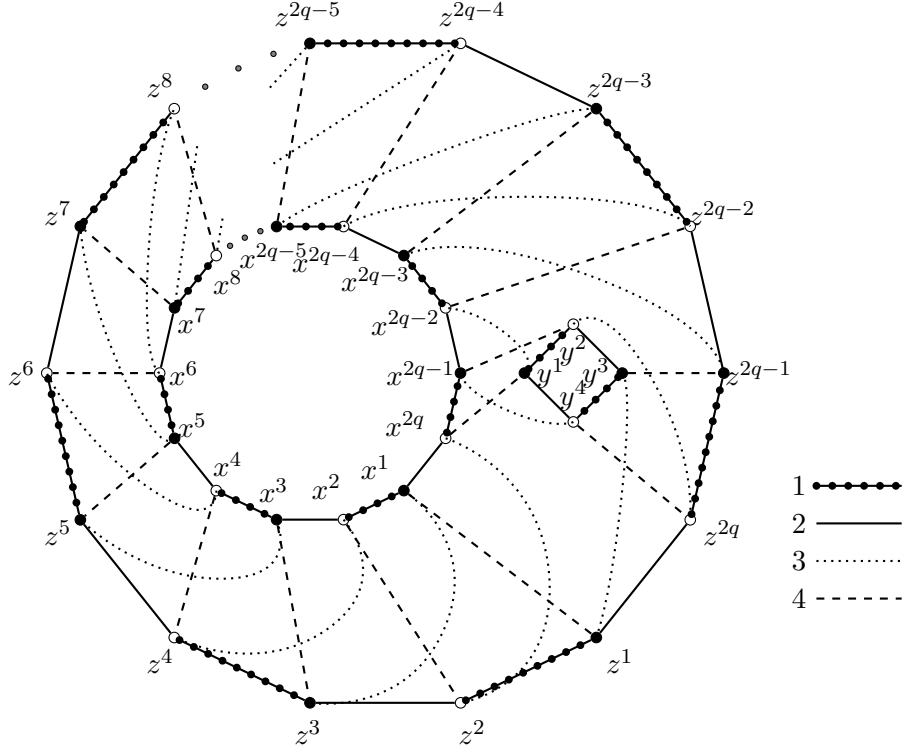


\smallskip

\noindent {\boldmath{$k=2$}} {\bf case:} The crystallization $\mathcal{M}_{2,q}$  is given in Fig. \ref{fig:M2q}.
Then, the components of $\Gamma^2_{12}$ are $G_{1} = C_{2q}(x^1, \dots, x^{2q})$, $G_2 = C_4 (y^1, \dots, y^4)$,
$G_3 = C_{2q}(z^1, \dots, z^{2q})$ and the components of $\Gamma^2_{34}$ are $H_1 = C_{2q}(y^1, x^{2q}, z^2, x^2,
\dots, z^{2q-2}$, $x^{2q-2})$, $H_2 = C_4(x^{2q-1}, y^2, z^{2q}, y^4)$, $H_3=C_{2q}(z^{2q-1}, y^3, z^1, x^1,
\dots, z^{2q-3}, x^{2q-3})$. Let $x$,  $y$ be the generators corresponding to $G_1$ and $G_2$ respectively. To
construct $\tilde{r}_1$ (resp., $\tilde{r}_2$) as in Eq. \eqref{tildar}, choose $(i, j)=(4,3)$ and $v_1= y^1$
(resp., $v_1=x^{2q-1}$). Then $H_1$ and $H_2$ represent the relations  $x^q y^{-1}$ and $y^2x^{-1}$ respectively.
Therefore, by Proposition \ref{prop:gagliardi79b}, $\pi(M_2,\ast) \cong \langle  x, y \, | \, x^q y^{-1},
y^2x^{-1} \rangle \cong \mathbb{Z}_{2q-1}$.

Let $T$ and  $T_2$ be the $3$-dimensional simplicial cell complexes represented by the color graphs
$\Gamma^2|_{\{x^1, x^2, x^3, z^3 \}}$ and $\Gamma^2|_{V(\Gamma^2)\setminus \{x^1, x^2, x^3, z^3 \}}$
respectively. Then $|T|$ and $|T_2|$ are solid tori and the facets (2-cells) of $T \cap T_2$ are $x^1_2, x^1_4,
x^2_3, x^2_4$, $x^3_1, x^3_3 ,z^3_1 ,z^3_2$. Thus, $|T \cap T_2|$ is a torus (see Fig. \ref{fig:M3q} (b)) with
$\pi_1(|T \cap T_2|,v_1) = \langle \alpha=[a], \beta=[b] \, | \, \alpha \beta \alpha^{-1} \beta^{-1}\rangle$,
where $a =x^2_{34}x^3_{34}$ and $b = x^3_{34}z^3_{13}z^3_{23}$. Then $b = x^1_{34}x^1_{13}x^1_{23} =
\partial(x^1_3) \sim 1 $ in $|T|$. Therefore, $\pi_1(|T|, v_1)=\langle \alpha, \beta\, | \, \beta  \rangle$.

Since $\alpha\beta=\beta\alpha$ in $|T| \cap |T_2|$, it follows that $ab \sim ba$ in $|T_2|$. Now $ab =
(x^2_{34}x^1_{34})(x^1_{34}x^1_{13}x^1_{23})$ $\sim x^2_{34}x^1_{13}x^1_{23} = x^2_{34}x^2_{13}x^1_{23} \sim
x^2_{23}x^1_{23}$. Therefore, $a^2b \sim aba \sim x^2_{23}x^1_{23}x^2_{34}x^1_{34}=
x^2_{23}x^{2q}_{23}x^{2q}_{34}x^1_{34} \sim x^2_{23}x^{2q}_{13}x^1_{34} = x^2_{23}z^1_{13}z^1_{34} \sim
x^2_{23}z^1_{23}$. Thus, $a^3b \sim x^2_{34}x^1_{34} x^2_{23}z^1_{23}$ $=x^2_{34}x^3_{34} x^3_{23}z^1_{23} \sim
x^2_{34}x^3_{13}z^1_{23}= x^4_{34}x^4_{13}z^1_{23} \sim x^4_{23}z^1_{23}$. Similarly, $a^{q-1}b \sim
x^{2q-4}_{23}z^1_{23}$. Therefore, $a^qb \sim x^2_{34}x^1_{34} x^{2q-4}_{23}z^1_{23} = x^2_{34}x^{2q-3}_{34}
x^{2q-3}_{23}z^1_{23} \sim x^2_{34}x^{2q-3}_{13}z^1_{23}= x^{2q-2}_{34}x^{2q-2}_{13}z^1_{23} \sim
x^{2q-2}_{23}z^1_{23}=x^{2q-1}_{23}z^1_{23} \sim x^{2q-1}_{34}x^{2q-1}_{13}z^1_{23}=
x^{2q-1}_{34}z^{1}_{13}z^1_{23} \sim x^{2q-1}_{34}z^{1}_{34}= x^{2q-1}_{34}z^{2q-1}_{34}$. Since $k=2$, we have
$z^{2q-1}_{13} = z^{2q}_{13}$. This implies,  $a^{2q-1}b^2 \sim x^{2q-1}_{34}z^{2q-1}_{34} x^{2q-4}_{23}z^1_{23}
= x^{2q-1}_{34}z^{2q-1}_{34} z^{2q-1}_{23}z^1_{23} \sim x^{2q-1}_{34}z^{2q-1}_{13}z^1_{23} =
z^{2q}_{34}z^{2q}_{13}z^{2q}_{23} = \partial (z^{2q}_3) \sim 1$ in $|T_2|$. Thus $\pi_1(|T_2|, v_1)=\langle
\alpha, \beta\, | \, \alpha^{2q-1} \beta^2, \alpha \beta \alpha^{-1} \beta^{-1}\rangle$. This implies that $|T
|\cup |T_2| =L(2q-1,2)$. Therefore, $\mathcal{M}_{2,q}$ is a crystallization of $L(2q-1,2) \cong L(2q-1,q)$.

\smallskip

\noindent {\boldmath{$k=3$}} {\bf case:} Here $z^{2q-1}_{13} \neq z^{2q}_{13}$. Let $\Gamma^3=(V(\Gamma^2)  \cup
\{y^5$,  $y^6$, $z^{2q+1}$, $z^{2q+2}\}$, $E(\Gamma^2) \setminus \{y^2z^{2q}$, $y^3z^{1}$, $y^3y^4$,
$z^{2q-1}z^{2q}\}\cup \{y^3y^5$, $y^5y^6$, $y^6y^4$, $z^{2q-1}z^{2q+1}$, $z^{2q+1} z^{2q+2}$, $z^{2q+2}z^{2q},
y^{5}z^{2q+1}$, $y^6z^{2q+2}$, $y^2z^{2q+1}$, $y^3z^{2q+2}$, $y^5z^{2q}$, $y^6z^{1}\} )$. Consider the following
coloring $\gamma^3$ on the edges of $\Gamma^3$: same colors on the old edges as in $\mathcal{M}_{2, q}$, color
$1$ on the edges $y^3y^5$, $y^6y^4$, $z^{2q-1}z^{2q+1}$, $z^{2q+2}z^{2q}$, color $2$ on the edges $y^5y^6$,
$z^{2q+1}z^{2q+2}$, color $3$ on the edges $y^2z^{2q+1}, y^3z^{2q+2}$, $y^5z^{2q}, y^6z^{1}$ and color $4$ on the
edges $y^{5}z^{2q+1}, y^6z^{2q+2}$  (see Fig.  \ref{fig:M3q} $(a)$). Let $T$ be as in the case $k=2$ and $T_3$ be
the cell complex represented by the colored graph $\Gamma^3|_{V(\Gamma^3)\setminus \{x^1, x^2, x^3, z^3 \}}$.

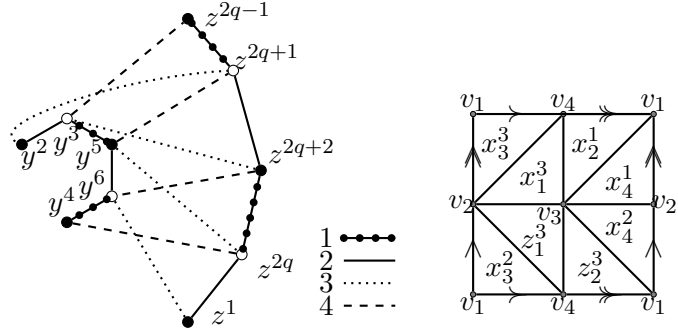
\begin{figure}[ht]
\tikzstyle{vert}=[circle, draw, fill=black!100, inner sep=0pt, minimum width=4pt]
\tikzstyle{vertex}=[circle, draw, fill=black!00, inner sep=0pt, minimum width=4pt]
\tikzstyle{ver}=[]
\tikzstyle{extra}=[circle, draw, fill=black!50, inner sep=0pt, minimum width=2pt]
\tikzstyle{edge} = [draw,thick,-]
\centering
\begin{tikzpicture}[scale=0.3]
\begin{scope}[shift={(2,-5.5)}]
\foreach \x/\y in {16/0,16/4,16/8}{
\node[extra] (x\y) at (\x,\y){};}
\foreach \x/\y in {20/0,20/4,20/8}{
\node[extra] (y\y) at (\x,\y){};}
\foreach \x/\y in {24/0,24/4,24/8}{
\node[extra] (z\y) at (\x,\y){};}
\foreach \x/\y in {x0/y0,x4/y4,x8/z8,y0/z0,y4/z4,y8/z8,x0/x4,x4/x8,y0/y4,y4/y8,z0/z4,z4/z8,x4/y8,y4/z8,y0/x4,z0/y4}{
\path[edge] (\x) -- (\y);}
\node[ver] () at (17.2,1.2){$x_3^2$};
\node[ver] () at (18.7,2.4){$z_1^3$};
\node[ver] () at (19.4,3.5){$v_3$};
\node[ver] () at (21.2,1.2){$z_2^3$};
\node[ver] () at (22.5,3){$x_4^2$};
\node[ver] () at (22.5,5){$x_4^1$};
\node[ver] () at (21,6.7){$x_2^1$};
\node[ver] () at (18.7,5.2){$x_1^3$};
\node[ver] () at (17,6.7){$x_3^3$};
\node[ver] () at (16,-0.5){$v_1$};
\node[ver] () at (20,-0.5){$v_4$};
\node[ver] () at (24,-0.5){$v_1$};
\node[ver] () at (16,8.5){$v_1$};
\node[ver] () at (20,8.5){$v_4$};
\node[ver] () at (24,8.5){$v_1$};
\node[ver] () at (15.5,4){$v_2$};
\node[ver] () at (24.5,4){$v_2$};
\node[ver] () at (18,0){$\succ$};
\node[ver] () at (22,0){$\succ$};
\node[ver] () at (22.2,0){$\succ$};
\node[ver] () at (18,8){$\succ$};
\node[ver] () at (22,8){$\succ$};
\node[ver] () at (22.2,8){$\succ$};
\node[ver] () at (16,2){$\wedge$};
\node[ver] () at (16,6){$\wedge$};
\node[ver] () at (16,6.2){$\wedge$};
\node[ver] () at (24,2){$\wedge$};
\node[ver] () at (24,6){$\wedge$};
\node[ver] () at (24,6.2){$\wedge$};
\end{scope}

\begin{scope}[shift={(16,-2)}, rotate=180]
\foreach \x/\y in {1/4,1/3,1/2,1/1}{
\node[ver] (d\y) at (\x,\y){};}
\foreach \x/\y in {4.5/4,4.5/3,4.5/2,4.5/1}{
\node[ver] (c\y) at (\x,\y){$\y$};}
\foreach \x/\y in {c1/d1,c2/d2}{
\path[edge] (\x) -- (\y);}
\path[edge, dashed] (c4) -- (d4);
\path[edge, dotted] (c3) -- (d3);
\draw [line width=3pt, line cap=round, dash pattern=on 0pt off 2\pgflinewidth]  (c1) -- (d1);
\end{scope}
\begin{scope}[rotate=-90]
\foreach \x/\y in {0/y^{4},120/y^{5},240/y^{2}}{
\node[ver] (\y) at (\x:1.4){$\y~$};
    \node[vert] (\y) at (\x:2.3){};
}
\foreach \x/\y in {60/y^{6},180/y^{3}}{
\node[ver] (\y) at (\x:1.3){$\y$};
    \node[vertex] (\y) at (\x:2.3){};
}
\node[ver] (b2) at (180:3){};
\end{scope}
\begin{scope}[]
\foreach \x/\y in {51.43/z^{2q-1},308.57/z^{1},360/z^{2q+2}}{
    \node[vert] (\y) at (\x:8.6){};
}
\foreach \x/\y in {45/z^{2q-1},5/z^{2q+2}}{
\node[ver] () at (\x:9.8){$~~~\y$};
}
\node[ver] () at (318:9.4){$z^{1}$};
\foreach \x/\y in {334.29/z^{2q},31/z^{2q+1}}{
\node[ver] (\y) at (\x:9.8){$~~\y$};
    \node[vertex] (\y) at (\x:8.6){};
}

\end{scope}
\foreach \x/\y in
{z^{2q}/z^{2q+2},y^{5}/y^{3},y^{6}/y^{4},z^{2q-1}/z^{2q+1},z^{2q}/z^{1},y^{2}/y^{3},y^{6}/y^{5},
z^{2q+2}/z^{2q+1}}{
    \path[edge] (\x) -- (\y);
}
\foreach \x/\y in {z^{2q}/z^{2q+2},y^{5}/y^{3},y^{6}/y^{4},z^{2q-1}/z^{2q+1}}{
    \draw [line width=3pt, line cap=round, dash pattern=on 0pt off 2\pgflinewidth]  (\x) -- (\y);
}
\foreach \x/\y in {y^{4}/z^{2q},z^{2q+2}/y^{6},y^{5}/z^{2q+1},y^{3}/z^{2q-1}}{
    \path[edge, dashed] (\x) -- (\y);
}
\foreach \x/\y in {z^{2q+2}/y^{3},y^{5}/z^{2q},y^{6}/z^{1}}{
    \path[edge, dotted] (\x) -- (\y);
}
\draw[edge, dotted] plot [smooth,tension=1.5] coordinates{(y^{2}) (b2) (z^{2q+1}) };
\end{tikzpicture}
\caption{($a$) Crystallization $\mathcal{M}_{3,q}$ of L$(3q-1,q)$ \hspace{8mm} ($b$) $|T_1 \cap
T_k|$\hspace{25mm}} \label{fig:M3q}
\end{figure}

Then, $a^{2q-1}b^2  \sim z^{2q}_{34}z^{2q-1}_{13}z^{2q}_{23} = z^{2q+1}_{34}z^{2q+1}_{13}z^{2q}_{23} \sim
z^{2q+1}_{23}z^{2q}_{23} =z^{2q+2}_{23}z^{2q}_{23}$. This implies, $a^{3q-1}b^3=(a^{q}b) (a^{2q-1} b^2) \sim
(x^{2q-1}_{34}z^{2q-1}_{34})(z^{2q+2}_{23}z^{2q}_{23})= z^{2q}_{34}z^{2q+2}_{34}z^{2q+2}_{23}z^{2q}_{23} \sim
z^{2q}_{34}z^{2q+2}_{13}z^{2q}_{23}$ $= z^{2q}_{34}z^{2q}_{13}z^{2q}_{23} = \partial (z^{2q}_3) \sim 1$ in
$|T_3|$. Thus, $\pi_1(|T_3|, v_1)=\langle  \alpha, \beta\, | \, \alpha^{3q-1} \beta^3, \alpha \beta \alpha^{-1}
\beta^{-1}\rangle$ and hence $|T |\cup |T_3| =L(3q-1,3)$. Therefore, $\mathcal{M}_{3,q}$ is a crystallization of
$L(3q-1,3) \cong L(3q-1,q)$.

\smallskip

\noindent {\boldmath{$k\geq 4$}} {\bf case:} Consider the graph $\Gamma^k=(V(\Gamma^{k-1})  \cup \{y^{2k-1},
y^{2k}, z^{2q+2k-5}, z^{2q+2k-4}\}$, $E(\Gamma^{k-1}) \setminus \{y^{2k-3}z^{2q}, y^{2k-2}z^{1}, y^{2k-2}y^4,
z^{2q+2k-6}z^{2q}\}\cup \{y^{2k-2}y^{2k-1}, y^{2k-1}y^{2k}, y^{2k}y^4, z^{2q+2k-6}z^{2q+2k-5}$, \linebreak
$z^{2q+2k-5}z^{2q+2k-4}$, $z^{2q+2k-4}z^{2q}$, $y^{2k-1}z^{2q+2k-5}$, $y^{2k}z^{2q+2k-4}$, $y^{2k-3}z^{2q+2k-5}$,
$y^{2k-2}z^{2q+2k-4}$, $y^{2k-1}z^{2q}$, $y^{2k}z^{1}\} )$. Also, consider the following coloring $\gamma^k$ on
the edges of $\Gamma^k$: same colors on the old edges as in $\mathcal{M}_{k-1,q}$, color $1$ on the edges
$y^{2k-2}y^{2k-1}$, $y^{2k}y^4$, $z^{2q+2k-6}z^{2q+2k-5}$, $z^{2q+2k-4}z^{2q}$, color $2$ on the edges
$y^{2k-1}y^{2k}$, $z^{2q+2k-5}z^{2q+2k-3}$, color $3$ on the edges \linebreak $y^{2k-3}z^{2q+2k-5}$,
$y^{2k-2}z^{2q+2k-4}$, $y^{2k-1}z^{2q}$, $y^{2k}z^{1}$ and color $4$ on the edges $y^{2k-1}z^{2q+2k-5}$,
$y^{2k}z^{2q+2k-4}$. Let $T$ be as in the case $k=2$ and $T_k$ be the cell complex represented by the colored
graph $\Gamma^k|_{V(\Gamma^k) \setminus \{x^1, x^2, x^3, z^3 \}}$.

\smallskip

\noindent {\em Claim.} $a^{kq-1}b^k \sim z^{2q}_{34}z^{2q+2k-4}_{13} z^{2q}_{23}$ in $|T_k|$.

\smallskip

We prove the claim by induction. It is true for $k=3$. Assume that $a^{(k-1)q-1}b^{k-1} \sim
z^{2q}_{34}z^{2q+2(k-1)-4}_{13} z^{2q}_{23}$ in $|T_{k-1}|$. Now, $a^{q}b \sim z^{2q}_{34}z^{1}_{34} =
z^{2q}_{34} z^{2q+2k-4}_{34}$ and $a^{(k-1)q-1}b^{(k-1)} \sim z^{2q}_{34} z^{2q+2k-6}_{13}z^{2q}_{23} =
z^{2q+2k-5}_{34} z^{2q+2k-5}_{13}z^{2q}_{23} \sim z^{2q+2k-5}_{23}z^{2q}_{23} = z^{2q+2k-4}_{23}z^{2q}_{23}$.
Thus, $a^{kq-1}b^{k} \sim (a^{q}b)(a^{(k-1)q-1}b^{(k-1)} \sim z^{2q}_{34}z^{2q+2k-4}_{34}z^{2q+2k-4}_{23}
z^{2q}_{23}=$ $z^{2q}_{34}z^{2q+2k-4}_{13}$ $z^{2q}_{23}$ in $|T_{k}|$. The claim now follows by induction.

Since $z^{2q}_{13} = z^{2q+2k-4}_{13}$ in $T_k$, by the claim we get  $a^{kq-1}b^{k} \sim 1$ in $|T_k|$. Thus,
$\pi_1(|T_k|, v_1)=\langle \alpha, \beta\, | \, \alpha^{kq-1} \beta^k, \alpha \beta \alpha^{-1}
\beta^{-1}\rangle$ and hence $|T |\cup |T_k| =L(kq-1,k)\cong L(kq-1,q)$. Therefore, $\mathcal{M}_{k,q}$ is a
crystallization of $L(kq-1,q)$.}
\end{eg}

\begin{eg}[{\bf \boldmath{$4(k+q)$}-vertex crystallization of \boldmath{$L(kq+1, q)$}}] \label{ex:Nkq}
{\rm Let $q\geq 4$. For each $k \geq 1$, we construct a $4(k+q)$-vertex 4-colored simple graph
$\mathcal{N}_{k,q}= (\Gamma^k, \gamma^k)$ with the color set $\{1, 2, 3, 4\}$ inductively which yields the
presentation $\langle  x, y \, | \, x^q y^{-1}, xy^k \rangle$. For this, we want $g^k_{12}=g^k_{34}=3$. Then,
without loss, $g^k_{13}=g^k_{24}=k+q-1$ and $g^k_{14}=g^k_{23}=k+q$, where $g^k_{ij}$ is the number of components
of $\Gamma^k_{ij}$ for $i \neq  j$. These imply, $\Gamma^k_{14}$ and $\Gamma^k_{23}$ must be union of
4-cycles and  $\Gamma^k_{13}$ (resp., $\Gamma^k_{24}$) has two 6-cycles and $(k+q-4)$ \, $4$-cycles. Then, by
Proposition \ref{prop:gagliardi79a}, $\mathcal{N}_{k,q}$ would be a crystallization of some connected closed
3-manifold $M_k$.

\noindent {\boldmath{$k=1$}} {\bf case:} The crystallization $\mathcal{N}_{1,q}$  is given in Fig. \ref{fig:N1q}.
Then, the components of $\Gamma^1_{12}$ are $G_{1} = C_{2q}(x^1, \dots, x^{2q})$, $G_2 = C_4 (y^1, \dots, y^4)$,
$G_3 = C_{2q}(z^1, \dots, z^{2q})$ and the components of $\Gamma^1_{34}$ are $H_1 = C_{2q}(y^3, x^{2}, z^2, x^4,
\dots, z^{2q-2}, x^{2q})$, $H_2 = C_4(z^{2q-1}, x^1, z^{1}, y^1)$, $H_3=C_{2q}(x^3, y^2, z^{2q}, y^4, x^{2q-1},
z^{2q-3}, x^{2q-3}, \dots, z^{4}, x^4, z^3)$. Let $x$,  $y$ be the generators corresponding to $G_1$ and $G_2$
respectively. To construct $\tilde{r}_1$ (resp., $\tilde{r}_2$) as in Eq. \eqref{tildar}, choose $(i, j)=(4,3)$
and $v_1= y^3$ (resp., $v_1=z^{2q-1}$). Then $H_1$ and $H_2$ represent the relations $x^q y^{-1}$ and $xy$
respectively. Therefore, by Proposition \ref{prop:gagliardi79b}, $\pi(M_1,\ast) \cong \langle  x, y \, | \, x^q
y^{-1}, xy \rangle \cong \mathbb{Z}_{q+1}$.


\begin{figure}[ht]

\tikzstyle{vert}=[circle, draw, fill=black!100, inner sep=0pt, minimum width=4pt] \tikzstyle{vertex}=[circle,
draw, fill=black!00, inner sep=0pt, minimum width=4pt] \tikzstyle{ver}=[] \tikzstyle{extra}=[circle, draw,
fill=black!50, inner sep=0pt, minimum width=2pt] \tikzstyle{edge} = [draw,thick,-]

\centering
\begin{tikzpicture}[scale=0.55]
\begin{scope}[shift={(-2,0)}]
\foreach \x/\y in
{51.43/x^{2q-2},102.86/x^{2q-4},154.28/x^{8},205.72/x^{6},257.14/x^{4},308.57/x^{2},360/x^{2q}}{ \node[ver] (\y)
at (\x:3.3){$\y~~$};
    \node[vertex] (\y) at (\x:4){};
} \foreach \x/\y in
{77.14/x^{2q-3},128.575/x^{9},180/x^{7},231.43/x^{5},282.86/x^{3},334.29/x^{1},25.72/x^{2q-1}}{ \node[ver] (\y)
at (\x:3.3){$\y~~$};
    \node[vert] (\y) at (\x:4){};
} \foreach \x/\y in {109.29/e1,115.72/e2,122.14/e3}{ \node[extra] (\y) at (\x:4){}; } \node[ver] (a^{7}) at
(110:6.8){}; \node[ver] (a^{8}) at (117:5){};
\end{scope}
\begin{scope}[shift={(5,0)}, rotate=-90]
\foreach \x/\y in {0/y^{2},180/y^{4}}{ \node[ver] (\y) at (\x:0.6){$\y$};
    \node[vertex] (\y) at (\x:1.4){};
} \foreach \x/\y in {90/y^{1},270/y^{3}}{ \node[ver] (\y) at (\x:0.7){$\y$};
    \node[vert] (\y) at (\x:1.4){};
} \node[ver] (a^{9}) at (180:3){}; \node[ver] (a^{10}) at (180:2){}; \node[ver] (b^{1}) at (320:3){};
\end{scope}
\begin{scope}[]
\foreach \x/\y in
{51.43/z^{2q-2},102.86/z^{2q-4},154.28/z^{8},205.72/z^{6},257.14/z^{4},308.57/z^{2},360/z^{2q}}{ \node[ver] (\y)
at (\x:9.6){$~~\y$};
    \node[vert] (\y) at (\x:9){};
} \foreach \x/\y in
{77.14/z^{2q-3},128.575/z^{9},180/z^{7},231.43/z^{5},282.86/z^{3},334.29/z^{1},25.72/z^{2q-1}}{ \node[ver] (\y)
at (\x:9.8){$~~\y$};
    \node[vertex] (\y) at (\x:9){};
}

\foreach \x/\y in {102.86/b^{2q-4},154.28/b^{8},205.72/b^{6},257.14/b^{4}}{

    \node[ver] (\y) at (\x:7){};
} \foreach \x/\y in {128.575/b^{9},180/b^{7},231.43/b^{5},360/b^{2q}}{

    \node[ver] (\y) at (\x:7){};
} \foreach \x/\y in {77.14/b^{2q-3},25.72/b^{2q-1},51.43/b^{2q-2},308.57/b^{2},282.86/b^{3}}{

    \node[ver] (\y) at (\x:6){};
} \foreach \x/\y in {109.29/e4,115.72/e5,122.14/e6}{ \node[extra] (\y) at (\x:9){}; } \node[ver] (a^{2q-4}) at
(122:7){}; \node[ver] (a^{2q-3}) at (133:6){}; \node[ver] (a^{2q-5}) at (115:8){}; \node[ver] (a^{2q-4}) at
(121:6.3){}; \node[ver] (4) at (11,-6){$4$}; \node[ver] (3) at (11,-5){$3$}; \node[ver](2) at (11,-4){$2$};
\node[ver](1) at (11,-3){$1$}; \node[ver] (8) at (14,-6){}; \node[ver](7) at (14,-5){}; \node[ver](6) at
(14,-4){}; \node[ver] (5) at (14,-3){};
\end{scope}
\foreach \x/\y in {x^{3}/x^{4},x^{5}/x^{6},x^{7}/x^{8},x^{2q-3}/x^{2q-2},x^{2q-1}/x^{2q},x^{1}/x^{2},
y^{4}/y^{3},y^{1}/y^{2},1/5,2/6}{
    \path[edge] (\x) -- (\y);
} \foreach \x/\y in
{x^{2}/x^{3},x^{4}/x^{5},x^{6}/x^{7},x^{9}/x^{8},x^{2q-4}/x^{2q-3},x^{2q-2}/x^{2q-1},x^{2q}/x^{1},
y^{1}/y^{4},y^{3}/y^{2}}{
    \path[edge] (\x) -- (\y);
} \foreach \x/\y in
{x^{2}/x^{3},x^{4}/x^{5},x^{6}/x^{7},x^{9}/x^{8},x^{2q-4}/x^{2q-3},x^{2q-2}/x^{2q-1},x^{2q}/x^{1},
y^{1}/y^{4},y^{3}/y^{2}}{
    \draw [line width=3pt, line cap=round, dash pattern=on 0pt off 2\pgflinewidth]  (\x) -- (\y);
} \foreach \x/\y in {z^{4}/z^{3},z^{6}/z^{5},z^{8}/z^{7},z^{2q-2}/z^{2q-3},z^{2q}/z^{2q-1},z^{2}/z^{1}}{
    \path[edge] (\x) -- (\y);
} \foreach \x/\y in {z^{2}/z^{3},z^{4}/z^{5},z^{6}/z^{7},z^{9}/z^{8},z^{2q-4}/z^{2q-3},z^{2q-2}/z^{2q-1},
z^{2q}/z^{1}}{
    \path[edge] (\x) -- (\y);
} \foreach \x/\y in {z^{2}/z^{3},z^{4}/z^{5},z^{6}/z^{7},z^{9}/z^{8},z^{2q-4}/z^{2q-3},z^{2q-2}/z^{2q-1},
z^{2q}/z^{1},1/5}{
    \draw [line width=3pt, line cap=round, dash pattern=on 0pt off 2\pgflinewidth]  (\x) -- (\y);
} \foreach \x/\y in
{x^{1}/z^{1},x^{2}/z^{2},x^{3}/z^{3},x^{4}/z^{4},x^{5}/z^{5},x^{6}/z^{6},x^{7}/z^{7},x^{8}/z^{8},
x^{9}/z^{9},x^{2q-4}/z^{2q-4},
x^{2q-3}/z^{2q-3},x^{2q-2}/z^{2q-2},x^{2q-1}/y^{4},x^{2q}/y^{3},y^{1}/z^{2q-1},y^{2}/z^{2q},3/7}{
    \path[edge, dotted] (\x) -- (\y);

} \foreach \x/\y in {z^{9}/a^{7},z^{8}/a^{8},x^{2q-3}/a^{2q-4},x^{2q-4}/a^{2q-3},y^{1}/z^{1},y^{4}/z^{2q},4/8}{
    \path[edge,  dashed] (\x) -- (\y);

} \foreach \x/\z/\y in
{z^{2}/b^{3}/x^{4},z^{3}/b^{4}/x^{5},z^{4}/b^{5}/x^{6},z^{5}/b^{6}/x^{7},z^{6}/b^{7}/x^{8}, z^{7}/b^{8}/x^{9},
z^{2q-4}/b^{2q-3}/x^{2q-2},z^{2q-3}/b^{2q-2}/x^{2q-1},y^{2}/b^{2}/x^{3},z^{2q-1}/a^{10}/x^{1},y^{3}/b^{1}/x^{2},
z^{2q-2}/a^{9}/x^{2q}} {

\draw[edge, dashed] plot [smooth,tension=1.5] coordinates{(\x) (\z) (\y) }; }
\end{tikzpicture}
\vspace{-4mm} \caption{Crystallization $\mathcal{N}_{1,q}$ of $L(q+1,q)$}\label{fig:N1q}
\end{figure}
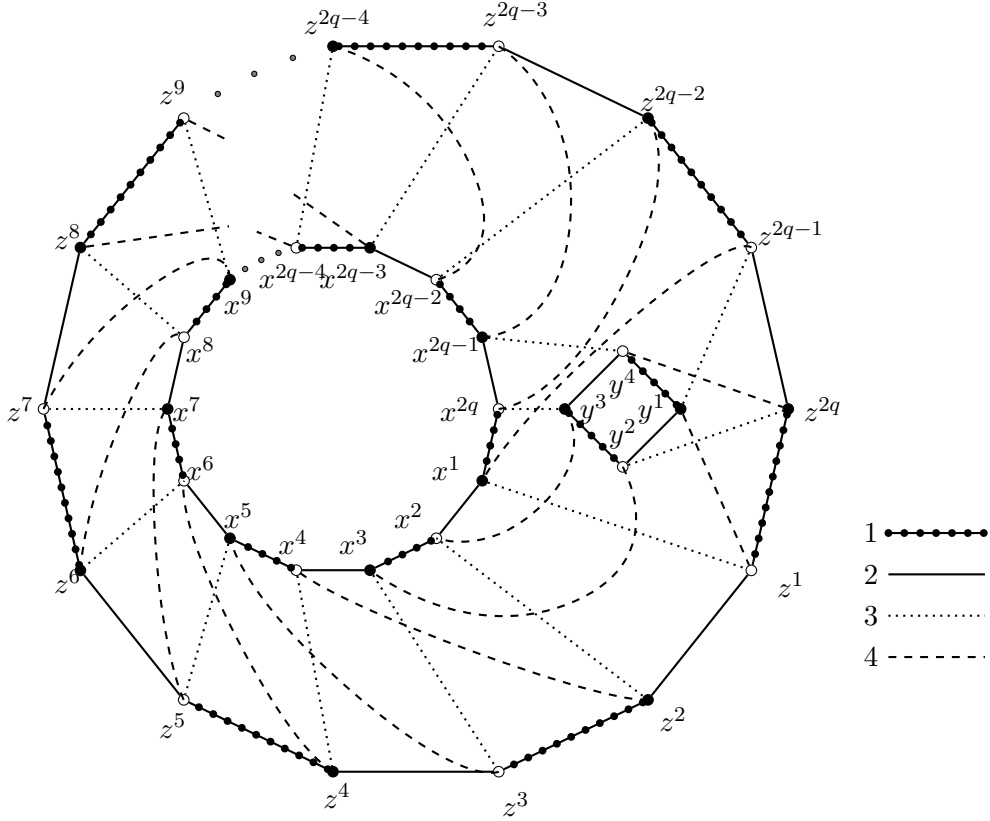

Let $T$ and  $T_1$ be the $3$-dimensional simplicial cell complexes represented by the color graphs
$\Gamma^1|_{\{x^5, x^4, x^3, z^3\}}$ and $\Gamma^1|_{V(\Gamma^1)\setminus \{ x^5, x^4, x^3, z^3 \}}$
respectively. Then $|T|$ and $|T_1|$ are solid tori and the facets (2-cells) of $T \cap T_1$ are $x^5_2$,
$x^5_3$, $x^4_3$, $x^4_4$, $x^3_1$, $x^3_4$, $z^3_1$, $z^3_2$. Thus, $|T \cap T_1|$ is a torus (see Fig.
\ref{fig:N2q} (b)) with $\pi_1(|T \cap T_1|,v_1) = \langle \alpha=[a], \beta=[b] \, | \, \alpha \beta \alpha^{-1}
\beta^{-1}\rangle$, where $a =x^4_{34}x^3_{34}$ and $b = x^3_{34}x^3_{13}x^4_{23}$. Then $b = x^3_{34}
x^3_{13}x^3_{23} = \partial(x^1_3) \sim 1 $ in $|T|$. Therefore, $\pi_1(|T|, v_1)=\langle \alpha, \beta\, | \,
\beta  \rangle$.

Since $\alpha\beta=\beta\alpha$ in $|T \cap T_1|$, it follows that $ab \sim ba$ in $|T_1|$. Now, $ab =
(x^4_{34}x^3_{34})(x^3_{34}x^3_{13}x^3_{23})$ $\sim x^4_{34}x^3_{13}x^3_{23} = z^2_{34}z^2_{13}x^3_{23} \sim
z^2_{23}x^3_{23} = z^1_{23}x^3_{23}$. Thus, $a^2b \sim aba \sim (z^1_{23}x^3_{23})(x^4_{34}x^3_{34})=
z^1_{23}x^4_{23}x^4_{34}$ $x^3_{34} \sim z^1_{23}x^4_{13}x^3_{34} = z^1_{23}x^5_{13}x^5_{34} \sim
z^1_{23}x^5_{23} = z^1_{23}x^6_{23}$. Therefore, $a^3b  \sim aba^2 \sim z^1_{23}x^6_{23}x^4_{34}x^3_{34} =
z^1_{23}x^6_{23}x^6_{34}x^3_{34} \sim z^1_{23}x^6_{13}x^3_{34}= z^1_{23}x^7_{13}x^7_{34} \sim z^1_{23}x^7_{23} =
z^1_{23}x^8_{23}$. Similarly we get, $a^{q-2}b \sim z^1_{23}x^{2q-2}_{23}$. Thus, $a^{q-1}b \sim a^{q-2}ba \sim
z^1_{23}x^{2q-2}_{23}x^4_{34}x^3_{34} = z^1_{23}x^{2q-2}_{23}x^{2q-2}_{34}x^3_{34} \sim z^1_{23}x^{2q-2}_{13}
x^3_{34} = z^1_{23}x^{2q-1}_{13}x^{2q-1}_{34} \sim z^1_{23}x^{2q-1}_{23} = z^1_{23}x^{2q}_{23}$. Therefore,
$a^{q}b \sim z^1_{23}x^{2q}_{23}x^4_{34}x^3_{34} = z^1_{23}x^{2q}_{23} x^{2q}_{34}$ $x^3_{34}$ $\sim
z^1_{23}x^{2q}_{13}x^3_{34} = z^1_{23}y^{2}_{13}y^{2}_{34} \sim z^1_{23}y^{2}_{23} \sim z^1_{34}z^1_{13}
y^{2}_{23} = z^1_{34}y^2_{13}y^{2}_{23} \sim z^1_{34}y^2_{34} = z^1_{34}z^{2q}_{34}$. Again, $a =
x^4_{34}x^3_{34} = y^3_{34}y^2_{34} \sim y^3_{24}y^3_{14}y^2_{34} = y^3_{24}y^2_{14}y^2_{34} \sim
y^3_{24}y^2_{24} = y^3_{24}z^1_{24}$. Therefore, $a^{q+1}b \sim aa^{q}b  \sim y^3_{24}z^1_{24}z^1_{34}z^{2q}_{34}
\sim$ $z^{2q}_{24}z^1_{14}z^{2q}_{34}= z^{2q}_{24}z^{2q}_{14}z^{2q}_{34} =\partial (z^{2q}_4) \sim 1$ in $|T_1|$.
Thus $\pi_1(|T_1|, v_1)=\langle  \alpha, \beta\, | \, \alpha^{q+1} \beta, \alpha \beta \alpha^{-1}
\beta^{-1}\rangle$. This implies that $|T |\cup |T_1| =L(q+1,1)$. Therefore, $\mathcal{N}_{1,q}$ is a
crystallization of $L(q+1,1) \cong L(q+1,q)$.

\smallskip

\noindent {\boldmath{$k=2$}} {\bf case:} Here $z^{1}_{14} \neq z^{2q}_{14}$. Let $\Gamma^2=(V(\Gamma^1)  \cup
\{y^5,  y^6, z^{2q+1} ,z^{2q+2}\}, E(\Gamma^1) \setminus \{y^2z^{2q}, y^1z^{2q-1}$, $y^1y^4$, $z^{1}z^{2q}\}$
$\cup \{y^1y^6$, $y^5y^6$, $y^4y^5$, $z^{1}z^{2q+2}$, $z^{2q+1}z^{2q+2}$, $z^{2q}z^{2q+1}$, $y^{5}z^{2q+1}$,
$y^6z^{2q+2}$, $y^2z^{2q+2}$, $y^1z^{2q+1}$, $y^6z^{2q}$, $y^5z^{2q-1}\})$. To construct $\mathcal{N}_{2,q}$,
consider the following coloring $\gamma^2$ on the edges of $\Gamma^2$: same colors on the old edges as in
$\mathcal{N}_{1,q}$, color $1$ on the edges $y^1y^6,y^4y^5$, $z^{1}z^{2q+2}, z^{2q}z^{2q+1}$ color $2$ on the
edges $y^5y^6$, $z^{2q+1}z^{2q+2}$, color $3$ on the edges $y^2z^{2q+2}$, $y^1z^{2q+1}$, $y^6z^{2q}$,
$y^5z^{2q-1}$ and color $4$ on the edges $y^{5}z^{2q+1}, y^6z^{2q+2}$  (see Fig. \ref{fig:N2q} $(a)$). Let $T$ be
as in the case $k=2$ and $T_2$ be the cell complex represented by the colored graph
$\Gamma^2|_{V(\Gamma^2)\setminus \{x^5, x^4, x^3, z^3\}}$.

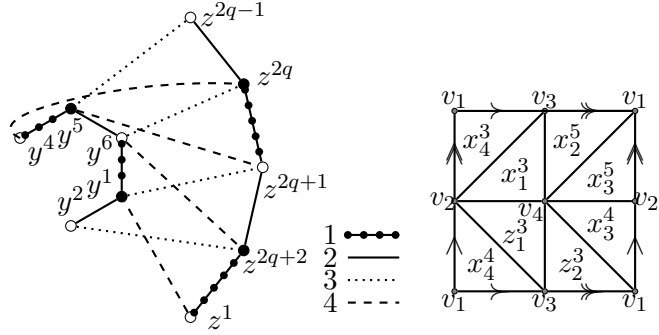
\begin{figure}[ht]
\tikzstyle{vert}=[circle, draw, fill=black!100, inner sep=0pt, minimum width=4pt]
\tikzstyle{vertex}=[circle, draw, fill=black!00, inner sep=0pt, minimum width=4pt]
\tikzstyle{ver}=[]
\tikzstyle{extra}=[circle, draw, fill=black!50, inner sep=0pt, minimum width=2pt]
\tikzstyle{edge} = [draw,thick,-]
\centering
\begin{tikzpicture}[scale=0.3]
\begin{scope}[shift={(1,-5.5)}]
\foreach \x/\y in {16/0,16/4,16/8}{
\node[extra] (x\y) at (\x,\y){};}
\foreach \x/\y in {20/0,20/4,20/8}{
\node[extra] (y\y) at (\x,\y){};}
\foreach \x/\y in {24/0,24/4,24/8}{
\node[extra] (z\y) at (\x,\y){};}
\foreach \x/\y in {x0/y0,x4/y4,x8/z8,y0/z0,y4/z4,y8/z8,x0/x4,x4/x8,y0/y4,y4/y8,z0/z4,z4/z8,x4/y8,y4/z8,y0/x4,z0/y4}{
\path[edge] (\x) -- (\y);}
\node[ver] () at (17.2,1.2){$x_4^4$};
\node[ver] () at (18.7,2.5){$z_1^3$};
\node[ver] () at (21.2,1.2){$z_2^3$};
\node[ver] () at (22.5,3){$x_3^4$};
\node[ver] () at (22.5,5){$x_3^5$};
\node[ver] () at (21,6.7){$x_2^5$};
\node[ver] () at (18.7,5.2){$x_1^3$};
\node[ver] () at (17,6.7){$x_4^3$};
\node[ver] () at (19.4,3.5){$v_4$};
\node[ver] () at (16,-0.5){$v_1$};
\node[ver] () at (20,-0.5){$v_3$};
\node[ver] () at (24,-0.5){$v_1$};
\node[ver] () at (16,8.5){$v_1$};
\node[ver] () at (20,8.5){$v_3$};
\node[ver] () at (24,8.5){$v_1$};
\node[ver] () at (15.5,4){$v_2$};
\node[ver] () at (24.5,4){$v_2$};
\node[ver] () at (18,0){$\succ$};
\node[ver] () at (22,0){$\succ$};
\node[ver] () at (22.2,0){$\succ$};
\node[ver] () at (18,8){$\succ$};
\node[ver] () at (22,8){$\succ$};
\node[ver] () at (22.2,8){$\succ$};
\node[ver] () at (16,2){$\wedge$};
\node[ver] () at (16,6){$\wedge$};
\node[ver] () at (16,6.2){$\wedge$};
\node[ver] () at (24,2){$\wedge$};
\node[ver] () at (24,6){$\wedge$};
\node[ver] () at (24,6.2){$\wedge$};
\end{scope}

\begin{scope}[shift={(16,-2)}, rotate=180]
\foreach \x/\y in {1/4,1/3,1/2,1/1}{
\node[ver] (d\y) at (\x,\y){};}
\foreach \x/\y in {4.5/4,4.5/3,4.5/2,4.5/1}{
\node[ver] (c\y) at (\x,\y){$\y$};}
\foreach \x/\y in {c1/d1,c2/d2}{
\path[edge] (\x) -- (\y);}
\path[edge, dashed] (c4) -- (d4);
\path[edge, dotted] (c3) -- (d3);
\draw [line width=3pt, line cap=round, dash pattern=on 0pt off 2\pgflinewidth]  (c1) -- (d1);
\end{scope}
\begin{scope}[rotate=-90]
\foreach \x/\y in {0/y^{2},120/y^{6},240/y^{4}}{
\node[ver] (\y) at (\x:1.5){$\y$};
    \node[vertex] (\y) at (\x:2.6){};
}
\foreach \x/\y in {60/y^{1},180/y^{5}}{
\node[ver] (\y) at (\x:1.5){$\y$};
    \node[vert] (\y) at (\x:2.6){};
}
\node[ver] (b4) at (180:3){};
\end{scope}
\begin{scope}[]
\foreach \x/\y in {51.43/z^{2q-1},308.57/z^{1},360/z^{2q+1}}{
    \node[vertex] (\y) at (\x:8.5){};
}
\foreach \x/\y in {45/z^{2q-1},355/z^{2q+1}}{
\node[ver] () at (\x:9.6){$~~\y$};
}
\node[ver] () at (315:9.4){$z^1$};
\foreach \x/\y in {334.29/z^{2q+2},25.72/z^{2q}}{
\node[ver] (\y) at (\x:9.6){$~~\y$};
    \node[vert] (\y) at (\x:8.5){};
}
\end{scope}
\foreach \x/\y in
{z^{1}/z^{2q+2},y^{4}/y^{5},y^{6}/y^{5},y^{6}/y^{1},y^{1}/y^{2},z^{2q}/z^{2q-1},z^{2q}/z^{2q+1},
z^{2q+2}/z^{2q+1}}{
    \path[edge] (\x) -- (\y);
}
\foreach \x/\y in {z^{1}/z^{2q+2},y^{4}/y^{5},y^{6}/y^{1},z^{2q}/z^{2q+1}}{
    \draw [line width=3pt, line cap=round, dash pattern=on 0pt off 2\pgflinewidth]  (\x) -- (\y);
}
\foreach \x/\y in {y^{6}/z^{2q+2},z^{1}/y^{1},y^{5}/z^{2q+1}}{
    \path[edge, dashed] (\x) -- (\y);
}
\foreach \x/\y in {y^{2}/z^{2q+2},z^{2q-1}/y^{5},y^{1}/z^{2q+1},y^{6}/z^{2q}}{
    \path[edge, dotted] (\x) -- (\y);
}
\draw[edge, dashed] plot [smooth,tension=1.5] coordinates{(y^{4}) (b4) (z^{2q}) };
\end{tikzpicture}
\caption{(a) Crystallization  $\mathcal{N}_{2,q}$ of $L(2q+1,q)$ \hspace{10mm} (b) $|T|\cap |T_k|$ \hspace{20mm}}
\label{fig:N2q}
\end{figure}

Then, $a^{q+1}b \sim z^{2q}_{24}z^{1}_{14}z^{2q}_{34} = z^{2q}_{24}z^{2q+2}_{14}z^{2q+2}_{34} \sim$
$z^{2q}_{24}z^{2q+2}_{24}= z^{2q}_{24}z^{2q+1}_{24}$. This implies, $a^{2q+1}b^2$ $\sim a^{q+1}ba^{q}b  \sim
z^{2q}_{24}z^{2q+1}_{24}z^{1}_{34}z^{2q}_{34} = z^{2q}_{24}z^{2q+1}_{24}z^{2q+1}_{34}z^{2q}_{34} \sim
z^{2q}_{24}z^{2q+1}_{14}z^{2q}_{34} = z^{2q}_{24}z^{2q}_{14}z^{2q}_{34} = \partial (z^{2q}_{4}) \sim 1$ in
$|T_2|$. Thus, $\pi_1(|T_2|, v_1)=\langle  \alpha, \beta\, | \, \alpha^{2q+1} \beta^2, \alpha \beta \alpha^{-1}
\beta^{-1}\rangle$ and hence $|T |\cup |T_2| =L(2q+1,2)$. Therefore, $\mathcal{N}_{2,q}$ is a crystallization of
$L(2q+1,2) \cong L(2q+1,q)$.

\noindent {\boldmath{$k\geq 3$}} {\bf case:} Let $\Gamma^k=(V(\Gamma^{k-1})\cup\{y^{2k+1}, y^{2k+2}, z^{2q+2k-3},
z^{2q+2k-2}\}$, $E(\Gamma^{k-1})$ $\setminus \{y^{2k}z^{2q}$, $y^{2k-1}z^{2q-1}$, $y^{2k-1}y^4$,
$z^{2q+2k-5}z^{2q}\}\cup \{y^{2k-1}y^{2k+2}$, $y^{2k+1}y^{2k+2}$, $y^{4}y^{2k+1}$, $z^{2q+2k-5}z^{2q+2k-2}$,
$z^{2q+2k-3}z^{2q+2k-2}$, $z^{2q}z^{2q+2k-3}, y^{2k+1}z^{2q+2k-3}, y^{2k+2}z^{2q+2k-2}$, $y^{2k}z^{2q+2k-2}$,
$y^{2k-1}z^{2q+2k-3}$, $y^{2k+2}z^{2q}$, $y^{2k+1}z^{2q-1}\} )$. To construct $\mathcal{N}_{k,q}$, consider the
following coloring $\gamma^k$ on the edges of $\Gamma^k$: same colors on the old edges as in
$\mathcal{N}_{k-1,q}$, color $1$ on the edges $y^{2k-1}y^{2k+2}$, $y^{4}y^{2k+1}$,  $z^{2q+2k-5}z^{2q+2k-2}$,
$z^{2q}z^{2q+2k-3}$, color $2$ on the edges  $y^{2k+1}y^{2k+2}$, $z^{2q+2k-3}z^{2q+2k-2}$, color $3$ on the edges
$y^{2k}z^{2q+2k-2}$, $y^{2k-1}z^{2q+2k-3}$, $y^{2k+2}z^{2q}$, $y^{2k+1}z^{2q-1}$ and color $4$ on the edges
$y^{2k+1}z^{2q+2k-3}$, $y^{2k+2}z^{2q+2k-2}$. Let $T$ be as in the case $k=1$ and $T_k$ be the cell complex
represented by the colored graph $\Gamma^k|_{V(\Gamma^k)\setminus \{ x^5, x^4, x^3, z^3 \}}$.

\smallskip

\noindent {\em Claim.} $a^{kq+1}b^k \sim z^{2q}_{24}z^{2q+2k-3}_{14}z^{2q}_{34}$ in $|T_k|$.

\smallskip

We prove the claim by induction. It is true for $k=2$. Assume that $a^{(k-1)q+1}b^{k-1} \sim z^{2q}_{24}
z^{2q+2(k-1)-3}_{14}z^{2q}_{34}$ in $|T_{k-1}|$. Now $a^{q}b \sim z^{2q+2k-3}_{34}z^{2q}_{34}$ and $a^{q(k-1)+1}
b^{(k-1)} \sim z^{2q}_{24}z^{2q+2k-4}_{24}= z^{2q}_{24}z^{2q+2k-3}_{24}$. So, $a^{qk+1}b^{k} \sim (a^{q(k-1)+1}
b^{(k-1)} (a^{q}b) \sim z^{2q}_{24}z^{2q+2k-3}_{24}z^{2q+2k-3}_{34}z^{2q}_{34} =z^{2q}_{24}z^{2q+2k-3}_{14}$ in
$|T_{k}|$. The claim now follows by induction.

Since $z^{2q}_{14} = z^{2q+2k-3}_{14}$ in $T_k$, by the claim we get $a^{kq+1}b^{k} \sim 1$ in $|T_k|$. Thus,
$\pi_1(|T_k|, v_1)=\langle  \alpha, \beta\, | \, \alpha^{kq+1} \beta^k, \alpha \beta \alpha^{-1}
\beta^{-1}\rangle$ and hence $|T |\cup |T_k| =L(kq+1,k) \cong L(kq+1,q)$. Therefore, $\mathcal{N}_{k,q}$ is a
crystallization of $L(kq+1,q)$.}
\end{eg}

A few days after we posted the first version of this article in the arXiv (arXiv:1308.6137), Casali and
Cristofori posted an article on complexity of lens spaces \cite{cc13} in the arXiv \linebreak (arXiv:1309.5728).
In that paper, the authors constructed crystallizations of $L(p,q)$ with $4S(p,q)$ vertices, where $S(p,q)$
denotes the sum of all partial quotients in the expansion of $q/p$ as a regular continued fraction. In
particular, they have constructed $L(kq-1, q)$ with $4(k+q-1)$ vertices for $k, q\geq 2$ and $L(kq+1, q)$ with
$4(k+q)$ vertices for $k, q\geq 1$. Their constructions are different from ours.

\section{Proofs of Theorems \ref{theorem:T2}, \ref{theorem:T4} and
Corollary \ref{cor:h2>6m} }

\noindent {\em Proof of Theorem} \ref{theorem:T2}. From Example \ref{ex:Mkq}, we know that $|\mathcal{M}_{2,3}| =
L(5,3)= L(5,2)$. Part (i) now follows from Lemmas \ref{lemma:psi(G)}, \ref{unique:J12}, \dots, \ref{unique:J4}.

If $f_3(X) <8$ then, by Theorem \ref{theorem:T1}, $\psi(M) <8$ and hence $\psi(M) =2$. Then $\pi(M, \ast) =
\{0\}$  and hence, by Poincar\'{e} conjecture, $M = S^3$. Part (ii) now follows from Lemma \ref{unique:J12}.
\hfill $\Box$

\bigskip

\noindent {\em Proof of Corollary} \ref{cor:h2>6m}. From the proof of Lemma \ref{unique:J3}, $m(Q_8) =2$.
Therefore, by Corollary \ref{cor:h2>=6m} and Lemma \ref{lemma:psi(M)}, $h_2(S^3/Q_8) \geq \psi(S^3/Q_8) -2 =18-2
> 12 = 6m(S^3/Q_8)$.

Again, from Corollary \ref{cor:h2>=6m} and Lemma \ref{lemma:psi(M)}, $h_2(S^1\times S^1\times S^1) \geq \psi(S^1
\times S^1\times S^1) -2 =24-2 > 6\times 3 = 6m(S^1\times S^1\times S^1)$.

Finally, by Theorem \ref{theorem:T2} (ii) and Corollary \ref{cor:h2>=6m}, $h_2(L(p,q)) \geq \psi(L(p,q)) -2 > 8 -
2 = 6\times 1 = 6m(L(p,q))$ for $p\geq 3$. This completes the proof. \hfill $\Box$

\bigskip

\noindent {\em Proof of Theorem} \ref{theorem:T4}.
The result follows from Examples \ref{ex:Mkq} and \ref{ex:Nkq}. \hfill $\Box$

\bigskip

\noindent {\bf Acknowledgement:} This work is supported in part by UGC Centre for Advanced Studies. The first
author thanks CSIR, India for SPM Fellowship. The authors thank M. R. Casali and P. Cristofori for pointing out
an error in an earlier version of this paper.


\end{document}